\documentclass[12pt, reqno]{amsart}
\usepackage{floatrow}
\usepackage{amssymb, amsfonts, amsthm, amsmath, amscd}
\usepackage{graphicx}   
\usepackage[latin2]{inputenc}
\usepackage{t1enc}
\usepackage[mathscr]{eucal}
\usepackage{indentfirst}
\usepackage{pict2e}
\usepackage{epic}
\numberwithin{equation}{section}
\usepackage{stmaryrd}
\usepackage{scalerel,stackengine}
\usepackage[margin=2.5cm]{geometry}
\usepackage{mathtools}
\usepackage[colorlinks = true,
            linkcolor = blue,
            urlcolor  = blue,
            citecolor = blue,
            anchorcolor = blue]{hyperref}
\usepackage[titletoc]{appendix}
\usepackage[T1]{fontenc}
\usepackage[numbers,square]{natbib}
\usepackage{adjustbox}
\usepackage{dsfont}
\usepackage{times}
\usepackage{enumerate}
\usepackage{setspace}
\usepackage{leftidx}
\usepackage{multirow}

\usepackage[most]{tcolorbox}
\usepackage{pagecolor}
\definecolor{bananamania}{rgb}{0.98, 0.91, 0.71}
\usepackage{tikz-cd}
\usepackage{pgf,tikz}
\usetikzlibrary{math}
\usetikzlibrary{arrows,decorations.markings}
\usetikzlibrary{patterns,intersections}
\usetikzlibrary{calc}
\usetikzlibrary{decorations.pathreplacing}

\makeatletter
\DeclareRobustCommand\bigop[1]{%
  \mathop{\vphantom{\sum}\mathpalette\bigop@{#1}}\slimits@
}
\newcommand{\bigop@}[2]{%
  \vcenter{%
    \sbox\z@{$#1\sum$}%
    \hbox{\resizebox{\ifx#1\displaystyle.9\fi\dimexpr\ht\z@+\dp\z@}{!}{$\m@th#2$}}%
  }%
}
\makeatother
\newenvironment{myproof}[2]{\paragraph{\textit{Proof of {#1} }{#2}.}}{\hfill$\square$}

\theoremstyle{plain}
\newtheorem{theorem}{Theorem}[section]
\newtheorem{lemma}[theorem]{Lemma}
\newtheorem{corollary}[theorem]{Corollary}
\newtheorem{proposition}[theorem]{Proposition}

\theoremstyle{definition}
\newtheorem{definition}[theorem]{Definition}

\newtheorem{remark}[theorem]{Remark}

\def\CC{\mathbb{C}}
\def\NN{\mathbb{N}}

\def\ZZ{\mathbb{Z}}
\def\PP{\mathbb{P}}
\def\ds{\displaystyle}

\def\a{\alpha}
\def\b{\beta}
\def\g{\gamma}
\def\s{\sigma}

\def\G{\Gamma}
\def\w{\omega}
\def\llll{\lambda}
\def\ra{\rightarrow}

\def\O{\Omega}

\def\LL{\mathcal{L}}

\def\BB{\mathcal{B}}
\def\GG{\mathcal{G}}
\def\FF{\mathcal{F}}
\def\PPc{\mathcal{P}}
\def\MM{\overline{\mathcal{M}}}
\def\OO{\mathcal{O}}

\def\gg{\mathfrak{g}}
\def\hh{\mathfrak{h}}

\def\ol{\overline}
\def\scl{\OO}
\def\ascl{\OO}
\def\localb{\OO}

\def\Q{Q^{\vee}}

\newcommand{\fibb}[1]{P_{#1}(G/B^+)}
\newcommand{\fibc}[1]{P_{#1}(G/P)}
\newcommand{\fib}[1]{P_{#1}(G/P)}

\def\mr{\mathring}

\def\ct{\CC^{\times}_t}
\def\st{S^1_t}
\def\cm{\CC^{\times}_{\mu}}
\def\cw{\CC^{\times}_{w(\lambda)}}
\def\wl{wt_{\lambda}}
\def\wll{w(\lambda)}
\newcommand{\fm}[2]{F^{attr}_{#1}(#2)}

\def\Gw{\G_{wt_{\lambda}}}
\newcommand{\Gt}[1]{G\times_{#1}}
\def\MMw{\MM(\wl,\b)}
\def\MMm{\MM(\mu,\b)}
\def\ec{\Lambda}
\def\rep{R(T)}
\def\agp{A_{G/P}}
\def\vp{\varphi}
\def\XXX{\mathcal{X}}
\def\YYY{\mathcal{Y}}

\def\aaa{\gamma}

\def\AA{\mathbb{A}^1}
\def\dfib{P_{\mu_1,\mu_2}}
\def\afib{P_{\mu_1}}
\def\bfib{P_{\mu_2}}
\def\ifib{P_{\mu_i}}
\def\abc{\beta}
\def\AMM{\mathfrak{M}}
\def\ac{a}
\def\AMMo{\mathfrak{M}^{\leqslant 1}_{0,0}}
\def\DDD{\mathcal{D}}

\def\CCC{\mathfrak{C}}

\DeclareMathOperator{\lie}{Lie}
\DeclareMathOperator{\fof}{Frac}
\DeclareMathOperator{\pr}{pr} 
\DeclareMathOperator{\pgw}{PKGW}
\DeclareMathOperator{\gw}{KGW}
\DeclareMathOperator{\ev}{ev}
\DeclareMathOperator{\id}{id}

\DeclareMathOperator{\aut}{Aut}
\DeclareMathOperator{\spann}{Span}

\DeclareMathOperator{\Ad}{Ad}
\DeclareMathOperator{\spec}{Spec}
\DeclareMathOperator{\sym}{Sym}

\DeclareMathOperator{\pic}{Pic}
\makeatletter
\newcommand*\bulletsmall{\mathpalette\bulletsmall@{.5}}
\newcommand*\bulletsmall@[2]{\mathbin{\vcenter{\hbox{\scalebox{#2}{$\m@th#1\bullet$}}}}}
\makeatother

\begin{document}
\title[Quantum $K$-theory of $G/P$ and $K$-homology of affine Grassmannian]{Quantum $K$-theory of $G/P$ and $K$-homology of affine Grassmannian}

\author{Chi Hong Chow}

\author{Naichung Conan Leung}

\begin{abstract} 
This paper is the $K$-theoretic analogue of a recent new proof, given by the first named author, of Peterson-Lam-Shimozono's theorem via Savelyev's generalization of Seidel representations. The outcome is a new proof of Lam-Li-Mihalcea-Shimozono's conjecture, including its extension to the parabolic case, which was first verified by Kato.
\end{abstract}



\maketitle

\section{introduction}\label{1}
Let $G$ be a simple and simply connected complex Lie group. In Schubert calculus, an unpublished result of Peterson \cite{Peter}, first proved by Lam-Shimozono \cite{LS}, states that the Pontryagin product of the homology of the affine Grassmannian $Gr_G$ determines completely, via an explicit ring homomorphism defined in terms of the (affine) Schubert bases, the quantum cup product of the quantum cohomology of any flag variety $G/P$. Recently, Chow \cite{me} has given a new proof of this result by computing Savelyev's parametrized version \cite{S_QCC} of Seidel representations \cite{Seidel}.

In $K$-theory, Lam-Li-Mihalcea-Shimozono \cite{LLMS} conjectured a similar homomorphism for the case of $G/B$ where the bases are replaced by the structure sheaves of the (affine) Schubert varieties. Their conjecture was first proved by Kato \cite{Kato1, Kato2}. Kato also extended the result to the general parabolic case \cite{Kato3}. In this paper, we give an alternative proof of Kato's result by following the approach in \cite{me}.
\begin{theorem} \label{main} The map
\[
\begin{array}{ccccc}
\Phi&:&K^T(Gr_G)&\ra &QK_T(G/P)[\ec^{-1}]\\ [0.5em]
& &  \ascl_{\wl} & \mapsto & q^{\llll+\Q_P}\scl_{\widetilde{w}}
\end{array}
\]
is an $\rep$-algebra homomorphism, where $\widetilde{w}$ is the minimal length coset representative of $wW_P$.
\end{theorem}

As already pointed out by Lam-Li-Mihalcea-Shimozono, Theorem \ref{main} implies immediately the finiteness property of the quantum $K$-product $\star$. 
\begin{corollary} \label{introcor1} For any $v_1, v_2\in W^P$, 
\[ \scl_{v_1}\star \scl_{v_2} \in K_T(G/P)\otimes \ZZ[\ec] .\]
\end{corollary}
\noindent Corollary \ref{introcor1} is not as obvious as the case of quantum cohomology because the moduli spaces of stable maps of arbitrary dimension contribute. What's more, as proved by Givental \cite{Givental}, in order for $\star$ to be associative, it is necessary to introduce a deformation of the Poincar\'e pairing $\chi_{G/P}(-\otimes -)$ by two-pointed $K$-theoretic GW invariants which is in general an infinite sum. Thus, the finiteness must follow from a non-trivial cancellation of terms. This issue has already been settled by Kato \cite{Kato1, Kato3, Kato2} and Anderson-Chen-Tseng \cite{ACT2}. See also the earlier work \cite{ACT1} of the authors of \cite{ACT2} and the work of Buch-Chaput-Mihalcea-Perrin \cite{BCMP1, BCMP2}.

\section*{Outline of the proof}
Our approach is to define a map by Gromov-Witten theory and show that it is an $\rep$-algebra homomorphism and has the desired form.

We first recall the proof of Peterson's original result given in \cite{me}. Take $Gr_G$ to be Pressley-Segal's model \cite{Segal}. In \textit{loc. cit.}, they constructed, for any holomorphic map $f:\G\ra Gr_G$, a holomorphic principal $G$-bundle $P_f$ over $\PP^1\times\G$ with a trivialization over $(\PP^1\setminus\{|z|\leqslant 1\})\times\G$. In particular, we obtain a $G/P$-fibration $\fib{f}$ over $\PP^1\times\G$ by reduction of fibers. One should think of $\fib{f}$ as a holomorphic family of $G/P$-fibrations over $\PP^1$ parametrized by $\G$. For any section class $\b$, define
\begin{center}
$\MM(f,\b):=~$moduli stack of holomorphic sections in $\fib{f}$ representing $\b$
\end{center}
and 
\[\ev:\MM(f,\b)\ra G/P\]
to be the evaluation map at $\infty\in\PP^1$. Thanks to the above trivialization, $\ev$ is well-defined.

For our purpose, we consider two classes of $f$: the $T$-fixed points of $Gr_G$ and Bott-Samelson resolutions of the affine Schubert varieties. They give rise to the localization basis $\{\eta_{\mu}\}_{\mu\in\Q}$ and the affine Schubert basis $\{\xi_{\wl}\}_{\wl\in W_{af}^-}$ of $H^T(Gr_G)$ respectively. Denote by $\MMm$ and $\MM(\wl,\b)$ the corresponding moduli stacks defined above. Define the \textit{Savelyev-Seidel homomorphism} \cite{S_QCC, Seidel}
\[ \Phi_{QH}: H^T(Gr_G)\ra QH_T(G/P)[\ec^{-1}]\]
either by 
\begin{equation}\label{introdef1}
\Phi_{QH}(\eta_{\mu}):= \sum_{\b}q^{\b} \ev_*[\MMm]^{vir}
\end{equation}
or 
\begin{equation}\label{introdef2}
\Phi_{QH}(\xi_{\wl}):= \sum_{\b}q^{\b} \ev_*[\MM(\wl,\b)]^{vir}.
\end{equation}
By the virtual localization formula \cite{GP}, these two definitions are equivalent. \eqref{introdef1} is used when we show that $\Phi_{QH}$ is a ring homomorphism and \eqref{introdef2} is used for the computation. The former follows from a degeneration argument and the latter relies on the following key observation:
\begin{center}
$\MMw$ is smooth and of expected dimension.
\end{center}
Since $\ev$ is $B$-equivariant, $\ev_*[\MMw]$ is equal to a multiple of a Schubert class or zero depending on whether the generic fiber of $\ev$ has zero or positive dimension. This reduces our computation to a purely combinatorial problem which can be solved in a straightforward way.

Back to the situation in the present paper, we will prove Theorem \ref{main} by adapting the above approach to the $K$-theoretic settings. By recalling the definition of the quantum $K$-product, one expects that the $K$-theoretic Savelyev-Seidel homomorphism should be defined by incorporating a new feature that the contribution of each $\MMm$ or $\MMw$ be corrected by the two-pointed $K$-theoretic GW invariants (see the paragraph following Corollary \ref{introcor1}). Define 
\[\agp:= \sum_{\b\ne 0}q^{\b}(\ev^{\b}_2)_*(\ev^{\b}_1)^*\]
where $\ev^{\b}_1, \ev^{\b}_2$ are the evaluation maps on $\MM_{0,2}(G/P,\b)$. The aforementioned correction is given by $(\id+\agp)^{-1}$. Therefore, we define
\[\Phi_{QK} : K^T(Gr_G)\ra QK_T(G/P)[\ec^{-1}]\]
by 
\[ \Phi_{QK}(\ascl_{\wl}):= (\id+\agp)^{-1}\left(\sum_{\b}q^{\b}\ev_*[\OO_{\MMw}]\right).\]
where $\ascl_{\wl}$ is the $K$-theoretic analogue of $\xi_{\wl}$. That $\Phi_{QK}$ is a ring homomorphism follows from similar localization and degeneration arguments as well as an argument used by Givental \cite{Givental} and Lee \cite{Lee} in their proof of the $K$-theoretic WDVV equation. 

The heart of the paper is the computation of $\Phi_{QK}(\ascl_{\wl})$. Our strategy is to introduce a $\CC^{\times}$-action on $\MMw$ by rescaling the domain of free loops in $G$, and apply Oprea's stacky version \cite{Oprea} of Bia\l{ynicki}-Birula's theorem \cite{BB} to this action. We show that if $\MMw\ne\emptyset$, there exists a unique component of $\MMw^{\CC^{\times}}$ whose Bia\l{ynicki}-Birula cell is open, and hence 
\begin{center}
$\MMw$ is either empty or irreducible.
\end{center}
A more in-depth analysis of this component gives
\[\{\b|~\MMw\ne \emptyset\} = [\llll]+\ec\]
and 
\begin{equation}\label{introeq}
\ev_*[\OO_{\MMw}] = (\ev^{\b-[\llll]}_2)_*(\ev^{\b-[\llll]}_1)^*\ev_*[\OO_{\MM(\wl,[\llll])}],\quad\b\in [\llll]+(\ec\setminus\{0\})
\end{equation}
where we put $[\llll]:=\llll+\Q_P$ for simplicity, $\ec$ is the semigroup of effective curve classes and $\ev^{\b-[\llll]}_1, \ev^{\b-[\llll]}_2:\MM_{0,2}(G/P,\b-[\llll])\ra G/P$ are the evaluation maps. Summing up \eqref{introeq} over all $\b$, weighted by $q^{\b}$, we get
\begin{align*}
\Phi_{QK}(\ascl_{\wl}) & = (\id +\agp)^{-1}\circ (\id +\agp)\left(q^{[\llll]} \ev_*[\OO_{\MM(\wl,[\llll])}]\right)\\
& = q^{[\llll]} \ev_*[\OO_{\MM(\wl,[\llll])}].
\end{align*}
The last term can easily be shown to be equal to $q^{[\llll]}\scl_{\widetilde{w}}$ as stated in Theorem \ref{main}.



\section{Preliminaries}\label{2}
\subsection{Some Lie-theoretic notations} \label{2aa}
Let $G$ be a simple and simply connected complex Lie group and $T\subset G$ a maximal torus. Put $\gg:=\lie(G)$ and $\hh:=\lie(T)$. We have the root space decomposition
\[\gg=\hh\oplus\bigoplus_{\a\in R}\gg_{\a}\]
where $R$ is the set of roots associated to the pair $(\gg,\hh)$ and each $\gg_{\a}$ is a one-dimensional eigenspace with respect to the adjoint action of $\hh$. Denote by $W$ the Weyl group. Fix a set $\{\a_1,\ldots,\a_r\}$ of simple roots of $R$. Denote by $\a_i^{\vee}$ the corresponding coroots. Define $R^+\subset R$ to be the set of positive roots spanned by $\a_1,\ldots,\a_r$. Define $B^-$ (resp. $B^+$) to be the Borel subgroup of $G$ containing $T$ with Lie algebra equal to $\hh\oplus\bigoplus_{\a\in -R^+}\gg_{\a}$ (resp. $\hh\oplus\bigoplus_{\a\in R^+}\gg_{\a}$).

Define the affine Weyl group $W_{af}:=W\ltimes\Q$ where $\Q\subset\hh$ is the $\ZZ$-span of $\a^{\vee}_i$, $i=1,\ldots,r$. Typical elements of $W_{af}$ are denoted by $\wl$. Define the affine simple roots $\widetilde{\a}_0,\ldots,\widetilde{\a}_r$ by $\widetilde{\a}_0:=-\a_0+1$ and $\widetilde{\a}_i:=\a_i$ for $i=1,\ldots, r$, where $\a_0$ is the highest positive root.
\subsection{Algebraic K-theory} \label{2ab}
A good reference for the following materials is \cite{Ginzburg}. 

Let $Y$ be a (finite-dimensional) scheme over $\CC$ with an action of a complex torus $T$ (which will be the maximal torus fixed in Section \ref{2aa} throughout). Define $K_T(Y)$ (resp. $K^T(Y)$) to be the Grothendieck group of $T$-equivariant vector bundles (resp. $T$-equivariant coherent sheaves) on $Y$. If $Y$ is smooth and quasi-projective, they are known to be isomorphic.

Pullback and tensor product of vector bundles give rise to the pullback operator and a ring structure on $K_T(Y)$ respectively. Tensor product also defines a $K_T(Y)$-module structure on $K^T(Y)$, since for any vector bundle $E$, $E\otimes -$ is an exact functor on the abelian category of coherent sheaves. If $f:Y\ra Z$ is a $T$-equivariant proper morphism, we define the pushforward operator
\[ f_*: K^T(Y)\ra K^T(Z)\]
by
\[ f_*([\mathcal{E}]):=\sum_{i\geqslant 0} (-1)^i [R^if_*(\mathcal{E})].\]
In particular, if $Y$ is proper and $Z$ is a point, the corresponding operator is denoted by $\chi_Y$. 

There is a natural $K_T(pt)$-module structure on $K_T(Y)$ and $K^T(Y)$, defined via the pullback operator associated to the structure morphism $Y\ra pt$. It is well-known that $K_T(pt)$ is isomorphic to the representation ring $\rep$ of $T$. We will adopt the latter notation throughout the paper.

\begin{remark} \label{2abrmk}
In this paper, we have to work with Deligne-Mumford stacks because moduli spaces of stable maps are not schemes in general. While the above definitions extend to Deligne-Mumford stacks, they are not strictly necessary for the computational aspect of this paper. We will bypass them by following the approach explained in \cite[Remark 5]{Lee}. 

Let $\YYY$ be a Deligne-Mumford stack arising from the moduli of stable maps to a smooth projective variety. Consider the canonical map
\[ p:\YYY\ra Y\]
from $\YYY$ to its coarse moduli $Y$. By the tameness property of $\YYY$ (see \cite{AV}), we have
\begin{equation}\label{2ab3}
p_*[\OO_{\YYY}]=[\OO_Y].
\end{equation}
The quantum $K$-invariants considered in this paper are of the form
\[\chi_{\YYY}(\OO^{vir}_{\YYY}\otimes\ev_1^*\a_1\otimes\cdots\otimes\ev_k^*\a_k)\]
where
\begin{itemize}
\item $\OO^{vir}_{\YYY}\in K^T(\YYY)$ is the virtual structure sheaf constructed in \cite{Lee};

\item $\ev_i$ are the evaluation maps on $\YYY$; and

\item $\a_i$ are some $K$-theory classes on the target space.
\end{itemize}
The key observation is that each $\ev_i$ factors through $p$, and hence the above invariant is equal to 
\[ \chi_Y(p_*\OO^{vir}_{\YYY}\otimes\ev_1^*\a_1\otimes\cdots\otimes\ev_k^*\a_k),\]
by the projection formula. For our computation, $\YYY$ will be smooth and of expected dimension. It follows that $\OO^{vir}_{\YYY}=[\OO_{\YYY}]$. This allows us to work only with the coarse moduli $Y$, by \eqref{2ab3}.
\end{remark}
\subsection{Flag varieties} \label{2b} A \textit{flag variety} is a homogeneous space $G/P$ where $P$ is any parabolic subgroup containing $B^+$. We have
\[ \lie(P)=\lie(B^+)\oplus\bigoplus_{\a\in -R_P^+}\gg_{\a}\]
where $R_P^+:=R_P\cap R^+$ and $R_P$ is the set of roots of $P$. Denote by $W_P$ the Weyl group of $P$ and by $W^P$ the set of minimal length coset representatives in $W/W_P$. For any $v\in W^P$, put $y_v:=vP\in G/P$. Then $\{y_v\}_{v\in W^P}$ is the set of $T$-fixed points of $G/P$. Define
\[ \scl_v := [\OO_{\ol{B^-\cdot y_v}}]\in K_T(G/P).\]
\begin{lemma}\label{2blemma}
$\{\scl_v\}_{v\in W^P}$ is an $\rep$-basis of $K_T(G/P)$.
\end{lemma}
\begin{proof}
See the proof of Lemma \ref{2cbasislemma}.
\end{proof}

We recall the \textit{equivariant quantum $K$-theory} of $G/P$ defined by Givental \cite{Givental}. See also the work of Lee \cite{Lee} which deals with general smooth projective varieties. Denote by $\ec\subset\pi_2(G/P)$ the semigroup of effective curve classes. We identify $\ec$ with $\bigoplus_{i=1}^r\ZZ_{\geqslant 0}\langle \a_i^{\vee}\rangle\subset\Q$ via the dual of the composition of three isomorphisms
\begin{equation}\label{2bisom}
\left(\Q/\Q_P\right)^*\xrightarrow{\rho~\mapsto L_{\rho}} \pic(G/P) \xrightarrow{c_1} H^2(G/P)\simeq \pi_2(G/P)^*
\end{equation}
where 
\begin{itemize}
\item $\Q_P:=\spann_{\ZZ}\{\a_i^{\vee}|~\a_i\in R_P\}$;
\item $L_{\rho}:= G\times_P \CC_{-\rho}$; and
\item $\CC_{-\rho}$ is the one-dimensional representation of weight $-\rho$ on which $P$ acts by forgetting the semi-simple and unipotent parts.
\end{itemize}

\noindent We have, as abelian groups, 
\[QK_T(G/P):= K_T(G/P)\otimes \ZZ[[\ec]]\]
where $\ZZ[[\ec]]$ is the formal completion of the group ring $\ZZ[\ec]$.

\begin{remark} The reason for enlarging the standard coefficient ring $\ZZ[\ec]$ for quantum cohomology is to ensure that the ring product we are going to define is well-defined. It turns out that this is unnecessary by Corollary \ref{introcor1}.
\end{remark}

What is non-trivial is the definition of the \textit{quantum $K$-product $\star$} on $QK_T(G/P)$. For any $\b\in\ec$ and $\aaa_1,\ldots,\aaa_k\in K_T(G/P)$, define
\[\gw^{\b}(\aaa_1,\ldots,\aaa_k) := \chi_{\MM_{0,k}(G/P,\b)}(\ev_1^*\aaa_1\otimes\cdots\otimes \ev_k^*\aaa_k)\in \rep\]
where $\MM_{0,k}(G/P,\b)$ is the Deligne-Mumford moduli stack of genus zero $k$-pointed stable maps to $G/P$ representing $\b$. Clearly, we can extend $\gw$ to a linear map $(QK_T(G/P))^{\otimes k}\ra \rep\otimes\ZZ[[\ec]]$ by linearity. Take an $\rep$-basis $\{e_i\}_{i\in I}$ of $K_T(G/P)$ (the Schubert basis, for example) and denote by $\{g^{ij}\}_{i,j\in I}$ the inverse of the matrix $\{\chi_{G/P}(e_i\otimes e_j)\}_{i,j\in I}$. It is well-known that the latter matrix is indeed invertible. Define a linear operator
\[\agp: QK_T(G/P)\ra QK_T(G/P)\]
by
\begin{equation}\nonumber
\agp(\aaa):=\sum_{i,j\in I}\sum_{\b\in\ec\setminus\{0\}}q^{\b}g^{ij}\gw^{\b}(\aaa,e_i) e_j.
\end{equation}
We have, for any $\aaa_1,\aaa_2\in QK_T(G/P)$,
\[\aaa_1\star\aaa_2 :=\sum_{i,j\in I}\sum_{\b\in\ec}q^{\b}g^{ij}\gw^{\b}(\aaa_1,\aaa_2,e_i)(\id +\agp)^{-1}(e_j). \]
By \cite{Givental} or \cite{Lee}, $\star$ defines a ring structure on $QK_T(G/P)$.

\subsection{Affine Grassmannian} \label{2c}
There are many models for the affine Grassmannian $Gr_G$. In this paper, we work with Pressley-Segal's version \cite{Segal}.

Define $H:=L^2(S^1;\gg)$ to be the Hilbert space of $L^2$-functions on $S^1$ with values in $\gg$. We have an orthogonal decomposition $H=H_+\oplus H_-$ where $H_+$ (resp. $H_-$) consists of functions whose negative (resp. non-negative) Fourier coefficients are zero. Let $\pr_{\pm}:H\ra H_{\pm}$ denote the orthogonal projections. Define
\[ Gr(H):=\{W\subset H\text{ closed subspaces}|~\pr_+|_W\text{ is Fredholm and }\pr_-|_W\text{ is Hilbert-Schmidt}\}.\]
It is proved in \cite{Segal} that $Gr(H)$ is a complex Hilbert manifold.

Fix a maximal compact subgroup $K$ of $G$. Define 
\begin{align*}
L_{sm}G & :=\{\text{smooth free loops in }G\}\\
L_{pol}G & :=\{\text{polynomial free loops in }G\}\\
\O_{sm}K & :=\{\text{smooth based loops in }K\}\\
\O_{pol}K & :=\{\text{polynomial based loops in }K\}.
\end{align*}
Then $L_{sm}G$ (resp. $\O_{sm}K$) is an infinite-dimensional complex (resp. real) Fr\'echet Lie group in the $C^{\infty}$-topology. Consider the following action on $H$ by $L_{sm}G$:
\[(\vp\cdot f)(z):= \Ad(\vp(z)) f(z)\quad \vp\in L_{sm}G,~f\in H\]
where $\Ad$ is the adjoint action. This action induces an $L_{sm}G$-action on $Gr(H)$ with respect to which the stabilizer of $H_+\in Gr(H)$ is equal to $L^0_{sm}G$, the subgroup of $L_{sm}G$ consisting of $\vp$ which extend to holomorphic functions defined on the unit disk.
\begin{theorem} \label{2cdiffeo} \cite[Theorem 8.6.2]{Segal} There exists a diffeomorphism 
\[ L_{sm}G\cdot H_+\simeq \O_{sm} K\]
under which the sub-orbit $L_{pol}G\cdot H_+$ corresponds to $\O_{pol}K$. 
\end{theorem}

From now on, we identify the orbit $L_{sm}G\cdot H_+$  (resp. $L_{pol}G\cdot H_+$) with $\O_{sm} K$ (resp. $\O_{pol} K$) via the above diffeomorphism . Since $G$ is assumed to be simply connected, $\O_{sm}K$ is connected and so lies in the connected component $Gr(H)^o$ of $Gr(H)$ containing $H_+$. For any natural number $n$, define
\[ Gr^{(n)}(H):=\{W\in Gr(H)^o|~z^n H_+\subseteq W\subseteq z^{-n} H_+\}\]
and 
\[\O_{pol}^{(n)}K:=\O_{pol}K\cap Gr^{(n)}(H).\]
Notice that $Gr^{(n)}(H)$ is a submanifold of $Gr(H)^o$ biholomorphic to the type-A Grassmannian $Gr(n\cdot \dim_{\CC}\gg, 2n\cdot \dim_{\CC}\gg)$ and $\O_{pol}^{(n)}K$ is a possibly singular closed subvariety of $Gr^{(n)}(H)$.
\begin{theorem} \label{2cunion}\cite[Theorem 8.3.3]{Segal} $\O_{pol}K=\bigcup_{n=0}^{\infty} \O_{pol}^{(n)}K$. 
\end{theorem}

Consider the action on $\O_{pol}K$ by the maximal torus $T\subset G$. It is easy to see that the fixed-point set $(\O_{pol}K)^T$ is equal to $\{x_{\mu}\}_{\mu\in\Q}$ where $x_{\mu}$ is the cocharacter of $T$ associated to any element $\mu\in\Q$. Define
\[\BB^{0,-}_{sm}:=\{\vp\in L^0_{sm}G|~\vp(0)\in B^-\}.\]
(By abuse of notation, the holomorphic extension of any $\vp\in L^0_{sm}G$ is denoted by the same symbol.)
\begin{theorem} \label{2cBruhat} \cite[Theorem 8.6.3]{Segal} 
\begin{enumerate}
\item (Bruhat decomposition) We have
\[ \O_{pol}K=\bigcup_{\mu\in\Q}\BB_{sm}^{0,-}\cdot x_{\mu}.\]
\item For any $\mu\in\Q$, the orbit $\BB_{sm}^{0,-}\cdot x_{\mu}$ is biholomorphic to a complex affine space.
\end{enumerate}
\end{theorem}

\begin{definition} Define the \textit{affine Grassmannian} $Gr_G:=\O_{pol}K$.
\end{definition}

Now we define, following \cite{KumarJEMS}, the \textit{$K$-homology} $K^T(Gr_G)$ of $Gr_G$. Notice that the definition does not follow directly from Section \ref{2ab} where we only deal with finite-dimensional schemes. By Theorem \ref{2cunion}, $Gr_G$ is the union of the chain of projective varieties
\[ \O_{pol}^{(0)}K\subset \O_{pol}^{(1)}K\subset \O_{pol}^{(2)}K\subset\cdots.\]
This chain induces a direct system of $\rep$-modules
\[ K^T(\O_{pol}^{(0)}K)\ra K^T(\O_{pol}^{(1)}K) \ra K^T(\O_{pol}^{(2)}K)\ra \cdots .\]
\begin{definition} Define
\[ K^T(Gr_G):=\varinjlim_{n} K^T(\O_{pol}^{(n)}K).\]
\end{definition}

Denote by $W_{af}^-$ the set of minimal length coset representatives in $W_{af}/W$. Notice that the map $W_{af}^-\ra \Q$ sending $\wl$ to $\wll$ is bijective.
\begin{definition} $~$
\begin{enumerate}
\item Let $\mu\in \Q$. Define
\[ \localb_{\mu}:= [\OO_{x_{\mu}}]\in K^T(Gr_G).\]
\item Let $\wl\in W_{af}^-$. Define
\[ \ascl_{\wl}:= [\OO_{\ol{\BB^{0,-}_{sm}\cdot x_{\wll}}}] \in K^T(Gr_G)\]
where $\ol{\BB^{0,-}_{sm}\cdot x_{\wll}}$ is the Zariski closure of $\BB^{0,-}_{sm}\cdot x_{\wll}$ taken in $\O_{pol}^{(n)}K$ for some large $n$.
\end{enumerate}
\end{definition}

\begin{lemma}\label{2cbasislemma} $~$
\begin{enumerate}
\item $\{\ascl_{\wl}\}_{\wl\in W_{af}^-}$ is an $\rep$-basis of $K^T(Gr_G)$. 
\item There exists a monomorphism 
\begin{equation}\label{2cmono}
K^T(Gr_G)\hookrightarrow \bigoplus_{\mu\in\Q}\fof(\rep)\langle\localb_{\mu}\rangle
\end{equation}
which fixes every $\localb_{\mu}$. 
\end{enumerate}
\end{lemma}
\begin{proof}
(1) follows from an argument of Kumar \cite{KumarJEMS} based on the following two standard results:
\begin{enumerate}[(i)]
\item (The excision sequence) If we have $U\xhookrightarrow{i} X \xhookleftarrow{j} X\setminus U$ where $X$ is projective and $U$ is open, then the sequence 
\begin{equation}\label{2cseq}
K^T(X\setminus U)\xrightarrow{j_*} K^T(X)\xrightarrow{i^*} K^T(U)\ra 0
\end{equation}
is exact.
\item (The Thom isomorphism: a special case) If $T$ acts on a vector space $\CC^r$ linearly, then $K^T(\CC^r)$ is freely generated by $[\OO_{\CC^r}]$. 
\end{enumerate}
Proofs of (i) and (ii) can be found in \cite{Ginzburg}. To prove (2), it suffices to show that every $\ascl_{\wl}$ is an $\fof(\rep)$-linear combination of $\localb_{\wll}$ and some other $\ascl_{w't_{\llll'}}$ with $\ell(w't_{\llll'})<\ell(\wl)$. This follows from \eqref{2cseq} and a local computation in $K^T(\CC^{\ell(\wl)})$.
\end{proof}

\begin{definition} Define an $\rep$-algebra structure $\bulletsmall$ on $\bigoplus_{\mu\in\Q}\fof(\rep)\langle\localb_{\mu}\rangle$ by 
\begin{equation}\label{2calg}
\localb_{\mu_1}\bulletsmall \localb_{\mu_2} := \localb_{\mu_1+\mu_2}\quad \mu_1,\mu_2\in \Q.
\end{equation}
\end{definition}

\begin{lemma} Via the monomorphism \eqref{2cmono}, $K^T(Gr_G)$ is a sub-$\rep$-algebra.
\end{lemma}
\begin{proof}
This is proved by Lam-Schilling-Shimozono \cite{LSS}. Notice that they first defined an $\rep$-algebra structure on $K^T(Gr_G)$ and verified \eqref{2calg}.

Alternatively, we show that it follows from our proof of Theorem \ref{main}, although the statement of this theorem assumes the lemma we are proving. Take $P=B^+$. In Section \ref{4a}, we will construct an $\rep$-algebra homomorphism
\[ \Phi: \bigoplus_{\mu\in\Q}\fof(\rep)\langle\localb_{\mu}\rangle \ra QK_T(G/B^+)[\ec^{-1}]\otimes\fof(\rep).\]
In Section \ref{4b}, \ref{4c} and \ref{4d}, we will show 
\[\Phi(\ascl_{\wl})=q^{\llll}\scl_w\]
where $\ascl_{\wl}$ is regarded as an element of the domain of $\Phi$ via \eqref{2cmono}. This implies that $\Phi$ sends an $\rep$-basis of $K^T(Gr_G)$ injectively into an $\rep$-basis of $K_T(G/B^+)\otimes\ZZ[\ec^{\pm}]$, and hence 
\[K^T(Gr_G)=\Phi^{-1}(QK_T(G/B^+)[\ec^{-1}])\] 
which is clearly a sub-$\rep$-algebra of $\bigoplus_{\mu\in\Q}\fof(\rep)\langle\localb_{\mu}\rangle$.
\end{proof}
\section{The key moduli}\label{3}
\subsection{The G/P-fibration} \label{3a}
The following theorem is the starting point of everything.
\begin{theorem} \cite[Theorem 8.10.2]{Segal} For any complex manifold $\G$, there exists a bijection between
\begin{enumerate}
\item the set of holomorphic maps $\G\ra \O_{sm}K$; and
\item the set of isomorphism classes of holomorphic principal $G$-bundles over $\PP^1\times \G$ with a trivialization over $(\PP^1\setminus\{|z|\leqslant 1\})\times \G$.
\end{enumerate}
\end{theorem}

We will need the $G/P$-bundle associated to the bundle in (2). To simplify the exposition on how this bundle is constructed, we introduce some Banach Lie groups as in \cite{me}. Define
\begin{align*}
D_0&:= \{z\in\CC|~|z|\leqslant 2\}\\
D_{\infty} &:= \{z\in\CC\cup\{\infty\}|~1/2\leqslant |z|\}\\
A&:= D_0\cap D_{\infty},
\end{align*}
and 
\begin{align*}
\GG&:= \{\vp:A\ra G|~\vp\text{ is continuous and }\vp|_{\mr{A}}\text{ is holomorphic}\}\\
\GG^0&:= \{\vp:D_0\ra G|~\vp\text{ is continuous and }\vp|_{\mr{D}_0}\text{ is holomorphic}\}\\
\GG^{\infty}&:= \{\vp:D_{\infty}\ra G|~\vp\text{ is continuous and }\vp|_{\mr{D}_{\infty}}\text{ is holomorphic}\}.
\end{align*}
These groups are complex Banach Lie groups in the $C^0$-topology. Moreover, $\GG^0$ and $\GG^{\infty}$ naturally embed into $\GG$ as subgroups in the sense of \cite{Bourbaki}. 

Define $\widetilde{\fib{}}$ to be the pushout of the diagram 
\begin{center}

\vspace{.2cm}
\begin{tikzpicture}
\tikzmath{\x1 = 3; \x2 = 2;}
\node (A) at (-\x1,\x2) {$\mr{D}_0\times \GG\times G/P$} ;
\node (B) at (\x1,\x2) {$\mr{D}_0\times \GG\times G/P$} ;
\node (C) at (0,0) {$\mr{A}\times \GG\times G/P$} ;
\path[left hook->, font=\tiny]
(C) edge node[left]{$\text{inclusion~}$} (A);
\path[right hook->, font=\tiny]
(C) edge node[right]{$~(z,\phi,y)\mapsto (z^{-1},\phi,\phi(z)\cdot y)$} (B);
\end{tikzpicture}.
\end{center}
We call the left (resp. right) copy $\mr{D}_0\times \GG\times G/P$ the \textit{$0$-chart} (resp. \textit{$\infty$-chart}). We have a map
\[ \widetilde{\pi}: \widetilde{\fib{}} \ra \PP^1\times \GG\]
defined by forgetting the factor $G/P$ in each of these charts. 

Define a left $\GG^{\infty}$-action and a right $\GG^0$-action on $\widetilde{\fib{}}$ by 
\[ \psi^{\infty}\cdot (z,\phi,y)\cdot \psi^0 := \left\{
\begin{array}{ll}
(z,\psi^{\infty}\phi\psi^0, \psi^0(z)^{-1}\cdot y)& 0\text{-chart}\\
(z,\psi^{\infty}\phi\psi^0, \psi^{\infty}(z^{-1})\cdot y)& \infty\text{-chart}\\
\end{array}
\right. \]
for any $\psi^0\in\GG^0$ and $\psi^{\infty}\in\GG^{\infty}$. It is easy to see that these actions are free and commute with each other. Moreover, $\widetilde{\pi}$ is equivariant where the action on $\PP^1$ is assumed to be trivial. Define
\[ \fib{} := \widetilde{\fib{}}/\GG^0\]
and 
\[ \pi:\fib{}\ra \PP^1\times (\GG/\GG^0)\]
to be the map induced by $\widetilde{\pi}$. It is straightforward to verify that 
\begin{enumerate}
\item the left $\GG^{\infty}$-action on $\widetilde{\fib{}}$ induces a left $\GG^{\infty}$-action on $\fib{}$;
\item $\pi$ is a $\GG^{\infty}$-equivariant $G/P$-fibration; and
\item the $\infty$-chart of $\widetilde{\fib{}}$ induces an $\infty$-chart of $\fib{}$ in the obvious sense.
\end{enumerate}

Observe that the above construction works not only for $G/P$ but also any $G$-spaces. In particular, if we take the $G$-equivariant line bundle $L_{\rho}=\Gt{P}\CC_{-\rho}$ associated to any $\rho\in(\Q/\Q_P)^*$ (see Section \ref{2b}), we obtain a line bundle $\LL_{\rho}$ on $\fib{}$ which restricts to $L_{\rho}$ on each fiber of $\pi$. Line bundles of this form will be useful in the next subsection.
\subsection{Definition of the moduli} \label{3b}
Following \cite{me}, we introduce two moduli spaces $\MMm$ and $\MMw$ which will be used for the definition and computation of the $\rep$-algebra homomorphism $\Phi$ stated in Theorem \ref{main} respectively. Recall the fibration $\fib{}$ defined in the last subsection.

Let $\mu\in \Q$. We have the associated cocharacter $x_{\mu}$ of $T$ which is naturally an element of $\GG$. By abuse of notation, the corresponding point in $\GG/\GG^0$ is denoted by the same symbol. Define 
\[\fibc{{\mu}}:=\fib{}|_{\PP^1\times\{x_{\mu}\}}\]
and 
\[ \pi_{\mu} : \fibc{{\mu}}\ra \PP^1\] 
to be the map induced by $\pi$. By \cite[Lemma 3.2]{me}, it is a smooth projective variety. Define $D_{\mu}:=\pi_{\mu}^{-1}(\infty)$ and $\iota_{\mu}:D_{\mu}\hookrightarrow \fibc{{\mu}}$ to be the inclusion.
\begin{definition}\label{3bdefm} Let $\b\in\pi_2(G/P)$.
\begin{enumerate}
\item Define 
\[ \MMm:= \bigcup_{\widetilde{\b}}\MM_{0,1}(\fibc{{\mu}},\widetilde{\b})\times_{(\ev_1,\iota_{\mu})}D_{\mu}\]
where $\widetilde{\b}$ runs over all classes in $\pi_2(\fibc{{\mu}})$ satisfying
\begin{enumerate}[(i)]
\item $(\pi_{\mu})_*\widetilde{\b}=[\PP^1]\in\pi_2(\PP^1)$; and 
\item $\langle\widetilde{\b},c_1(\LL_{\rho})\rangle = \langle\b,c_1(L_{\rho})\rangle$ for any $\rho\in(\Q/\Q_P)^*$.
\end{enumerate}
(The line bundles $\LL_{\rho}$ and $L_{\rho}$ are defined in Section \ref{3a} and Section \ref{2b} respectively.)

\vspace{.2cm}
\item Define 
\[\ev:\MMm\ra D_{\mu}\simeq G/P\]
to be the morphism induced by the evaluation map $\ev_1$ on $\MM_{0,1}(\fibc{{\mu}},\widetilde{\b})$.
\end{enumerate}
\end{definition}

Next we define $\MMw$. Define
\[ \BB^{0,-}:=\{\vp\in\GG^0|~\vp(0)\in B^-\}.\]
For any affine simple root $\widetilde{\a}_i$, $i=0,\ldots,r$ (see Section \ref{2aa}), there exists a unique connected subgroup $\PPc_{\widetilde{\a}_i}$ of $\GG$ with $\BB^{0,-}\subset \PPc_{\widetilde{\a}_i}$ such that 
\[ \lie(\PPc_{\widetilde{\a}_{i}}) = \left\{
\begin{array}{ll}
\lie(\BB^{0,-})\oplus z^{-1}\gg_{-\a_0}& i=0\\ [1em]
\lie(\BB^{0,-})\oplus \gg_{\a_i}& i=1,\ldots, r\\
\end{array}
 \right. .\]
For any $\wl\in W_{af}^-$, choose a reduced word decomposition $(i_1,\ldots,i_{\ell(\wl)})$ of it. Define the associated Bott-Samelson variety
\[\Gw:= \PPc_{\widetilde{\a}_{i_1}}\times_{\BB^{0,-}}\cdots\times_{\BB^{0,-}}\PPc_{\widetilde{\a}_{i_{\ell(\wl)}}}/ \BB^{0,-}.\]
It is easy to see that $\Gw$ is a smooth projective variety with a structure of iterated $\PP^1$-bundles. Define a holomorphic map
\[ f_{\wl}:\Gw\ra \GG/\GG^0\]
by
\[ f_{\wl}([\vp_1:\cdots:\vp_{\ell(\wl)}]) := \vp_1\cdots \vp_{\ell(\wl)}\GG^0.\]
Define 
\[\fib{\wl}:=(\PP^1\times\Gw)\times_{(\id\times f_{\wl},\pi)}\fib{}\]
and 
\[ \pi_{\wl}:\fib{\wl}\ra \PP^1\times \Gw\]
to be the map induced by $\pi$. By \cite[Lemma 3.2]{me}, $\fib{\wl}$ is a smooth projective variety. Define $D_{\wl}:=\pi_{\wl}^{-1}(\{\infty\}\times \Gw)$ and $\iota_{\wl}:D_{\wl}\hookrightarrow \fib{\wl}$ to be the inclusion. Then $D_{\wl}$ is a smooth divisor of $\fib{\wl}$ and isomorphic to $\Gw\times G/P$ via the $\infty$-chart of $\fib{\wl}$. 

\begin{definition}\label{3bdefw} Let $\b\in\pi_2(G/P)$.
\begin{enumerate}
\item Define 
\[ \MMw:= \bigcup_{\widetilde{\b}}\MM_{0,1}(\fib{\wl},\widetilde{\b})\times_{(\ev_1,\iota_{\wl})}D_{\wl}\]
where $\widetilde{\b}$ runs over all classes in $\pi_2(\fib{\wl})$ satisfying
\begin{enumerate}[(i)]
\item $(\pi_{\wl})_*\widetilde{\b}=[\PP^1\times\{pt\}]\in\pi_2(\PP^1\times\Gw)$; and 
\item $\langle\widetilde{\b},c_1(\LL_{\rho})\rangle = \langle\b,c_1(L_{\rho})\rangle$ for any $\rho\in(\Q/\Q_P)^*$.
\end{enumerate}

\vspace{.2cm}
\item Define 
\[\ev:\MMw\ra G/P\]
to be the composite
\[ \MMw\ra D_{\wl}\simeq \Gw\times G/P\ra G/P\]
where the first arrow is induced by the evaluation map $\ev_1$ on $\MM_{0,1}(\fib{\wl},\widetilde{\b})$ and the second arrow is the canonical projection. 
\end{enumerate}
\end{definition}

In order to compute $\Phi$, we have to establish some geometric properties of $\MMw$. This will be done in Section \ref{3e}. The intermediate subsections \ref{3c} and \ref{3d} will serve as preparations.
\subsection{An extra torus action on the moduli} \label{3c}
We define an algebraic $\CC^{\times}$-action on $\MMw$. To avoid confusion with other actions, we will introduce the subscript ``$t$'' for the action. 

The desired action is defined in several steps:
\begin{enumerate}
\item Define an $\st$-action on $\GG$ by
\[ (t\cdot \vp)(z) := \vp(tz)\quad t\in \st,~\vp\in\GG.\]
Observe that this action preserves the subgroups $\GG^0$, $\GG^{\infty}$, $\BB^{0,-}$ and $\PPc_{\widetilde{\a}_i}$.

\item Define an $\st$-action on $\widetilde{\fib{}}$ by
\[ t\cdot (z,\phi,y) = \left\{
\begin{array}{cl}
(t^{-1}z,t\cdot\phi,  y)& 0\text{-chart}\\
(tz,t\cdot \phi ,y) & \infty\text{-chart}
\end{array}
 \right. .\]
 
It is straightforward to check that this action descends to an $\st$-action on $\fib{}$  which satisfies the following properties:
 \begin{enumerate}[(i)]
 \item it is compatible with the $\GG^{\infty}$-action in the sense that 
 \[ t\cdot(\psi^{\infty}\cdot x) = (t\cdot\psi^{\infty})\cdot (t\cdot x)\]
 for any $t\in \st$, $\psi^{\infty}\in \GG^{\infty}$ and $x\in \fib{}$; and
 \item the map $\pi:\fib{}\ra\PP^1\times (\GG/\GG^0)$ is equivariant where $\st$ acts on $\PP^1$ by 
 \[ t\cdot z = \left\{
\begin{array}{cl}
t^{-1}z& 0\text{-chart}\\
tz & \infty\text{-chart}
\end{array}
 \right. .\]
\end{enumerate}

\item Consider the $\st$-action on $\Gw$ induced by the one defined in (1). Then $f_{\wl}$ is equivariant.

\item It follows that there is an induced $\st$-action on $\fib{\wl}$. It is not hard to show that this $\st$-action extends to a unique algebraic $\ct$-action. See Remark \ref{3crmk} below.

\item Clearly, $D_{\wl}$ is $\ct$-invariant, and hence the $\ct$-action in (4) induces a $\ct$-action on $\MMw$.
\end{enumerate}

\begin{remark}\label{3crmk} 
Every $S^1$-action on a smooth projective variety by biholomorphisms extends to a unique holomorphic $\CC^{\times}$-action. In this paper, we require the latter action to be algebraic. One way of showing this is to construct a lift of the given $S^1$-action on an ample line bundle. Observe that the spaces in question, including $\Gw$ and $\fib{\wl}$, have a structure of iterated fibrations such that the fibers at each step are Fano. Thus, it suffices to deal with the following situation. Let $F\ra E\xrightarrow{\pi} B$ be a fibration with $F$ being Fano. Suppose $S^1$ acts on $E$ and $B$ such that $\pi$ is equivariant and $B$ admits an equivariant ample line bundle $\mathcal{L}_B$. For sufficiently large $N$, the line bundle $\mathcal{L}_E:=\w_{E/B}^{\vee}\otimes\pi^{*}\left(\mathcal{L}_B^{\otimes N}\right)$ is ample. Then $S^1$ acts on $\mathcal{L}_E$ naturally because $\w_{E/B}^{\vee}$ is formed out of the vertical tangent bundle of $\pi$ and $\pi$ is equivariant. 
\end{remark}
 
\subsection{Constant sections} \label{3d}
 Let $\mu\in\Q$. Recall the point $x_{\mu}\in\GG/\GG^0$. Define 
 \[P_{\mu}:=\{g\in G|~g\cdot x_{\mu}=x_{\mu}\}.\]
 
\begin{lemma}
$P_{\mu}$ is a parabolic subgroup with Lie algebra
\begin{equation}\label{3deq}
\lie(P_{\mu}) = \hh\oplus \bigoplus_{\a(\mu)\leqslant 0} \gg_{\a}.
\end{equation}
\end{lemma}
\begin{proof}
First notice that $P_{\mu}$ is an algebraic subgroup of $G$. It is clear that the RHS of \eqref{3deq} is a parabolic subalgebra and so defines a parabolic subgroup $P'$. Then $P'\subseteq P_{\mu}$, and hence $P_{\mu}$ is connected. Suppose $P'\subsetneq P_{\mu}$. Then there exists $v\in\gg_{\a}\setminus\{0\}$ for some $\a\in R$ with $\a(\mu)>0$ such that $g:=\exp(v)\in P_{\mu}$. This implies that the holomorphic function $z\mapsto x_{\mu}(z^{-1})g x_{\mu}(z)$ extends to a holomorphic function on $\PP^1$. Since $G$ is affine, this function is constant but this is impossible. 
\end{proof}

Recall $\fibc{{\mu}}=\fib{}|_{\PP^1\times\{x_{\mu}\}}$. Observe it can also be defined as the pushout of the diagram
\begin{center}

\vspace{.2cm}
\begin{tikzpicture}
\tikzmath{\x1 = 3; \x2 = 2;}
\node (A) at (-\x1,\x2) {$\CC\times G/P$} ;
\node (B) at (\x1,\x2) {$\CC\times G/P$} ;
\node (C) at (0,0) {$\CC^{\times}\times G/P$} ;
\path[left hook->, font=\tiny]
(C) edge node[left]{$\text{inclusion~}$} (A);
\path[right hook->, font=\tiny]
(C) edge node[right]{$~(z,y)\mapsto (z^{-1},x_{\mu}(z)\cdot y)$} (B);
\end{tikzpicture}.
\end{center}

\noindent Notice $x_{\mu}\in (\GG/\GG^0)^{\st}$, and hence the $\st$-action preserves $\fibc{{\mu}}$. One checks easily that the induced action on $\fibc{{\mu}}$, which actually extends to a $\ct$-action, reads
\[ t\cdot (z,y) = \left\{
\begin{array}{cl}
(t^{-1}z,x_{\mu}(t)\cdot y)& 0\text{-chart}\\
(tz,y) & \infty\text{-chart}
\end{array}
 \right. .\]


Consider the $\CC^{\times}$-action on $G/P$ induced by the cocharacter $x_{\mu}:\CC^{\times}\ra T$. We will write $\cm$ in place of $\CC^{\times}$ in this context. Let $y\in G/P$. Since $G/P$ is complete, the morphism $\CC^{\times}\ra G/P:z\mapsto x_{\mu}(z^{-1})\cdot y$ extends to a morphism defined on $\CC$.

\begin{definition} \label{3dconstsectiondef}
Let $y\in G/P$. Define a section $u_y$ of $\fib{\mu}$ by 
\[ u_y(z) := \left\{
\begin{array}{cl}
(z,x_{\mu}(z^{-1})\cdot y)& 0\text{-chart}\\
(z,y) & \infty\text{-chart}
\end{array}
 \right. .\]
Any sections of the form $u_y$ are called \textit{constant sections} (meaning constant in the $\infty$-chart). 
\end{definition}

\begin{lemma}\label{3dconst} Let $u:\PP^1\ra \fibc{{\mu}}$ be a holomorphic section. Suppose for any $t\in \ct$, there exists $\phi\in\aut(\PP^1)$ such that for any $z\in\PP^1$,
\[ t\cdot u(z) = u(\phi(z)).\]
Then $u$ is a constant section.
\end{lemma}
\begin{proof}
Restricting $u$ to the $0$-chart and $\infty$-chart, we get two maps $u_0,u_{\infty}:\CC\ra G/P$ satisfying
\[ u_{\infty}(z^{-1}) = x_{\mu}(z)\cdot u_0(z)\quad\text{for any }z\in\CC^{\times}.\]
The given condition implies $u_{\infty}$ is constantly equal to a point $y\in G/P$. This forces $u\equiv u_y$.
\end{proof}

Denote by $\fm{\cm}{G/P}$ the unique component of $(G/P)^{\cm}$ whose normal bundle has only positive weights. The superscript ``attr'' will be explained in Definition \ref{3efix}.
\begin{lemma}\label{3dtrans} $P_{\mu}$ preserves and acts transitively on $\fm{\cm}{G/P}$.
\end{lemma}
\begin{proof}
It suffices to look at the infinitesimal action. Let $y\in \fm{\cm}{G/P}$. Notice that $T_y(G/P)$ is a direct sum of weight spaces (with respect to the $\cm$-action) of non-negative weights and the set of these weights (counted with multiplicities) is a subset of the set of weights of the $\cm$-module $\gg$. By definition, $\lie(P_{\mu})$ contains all non-positive weights of $\gg$. It follows that the weights contributed by the infinitesimal action of $P_{\mu}$ are precisely all the zero weights.
\end{proof}

\begin{definition}\label{3dbwl} Define $\b_{\mu}\in\pi_2(G/P)$ to be the unique element such that 
\[\deg(u^*\LL_{\rho})=\langle \b_{\mu}, c_1(L_{\rho})\rangle \]
for any $\rho\in (\Q/\Q_P)^*$ where $u$ is the constant section of $\fibc{{\mu}}$ corresponding to a point in $\fm{\cm}{G/P}$. (The line bundles $\LL_{\rho}$ and $L_{\rho}$ are defined in Section \ref{3a} and Section \ref{2b} respectively.)
\end{definition}
\subsection{Some properties of the moduli} \label{3e}
\begin{proposition} \cite[Proposition 4.5]{me} The stack $\MMw$ is smooth and of expected dimension.
\end{proposition}

In what follows, we prove some further properties of $\MMw$. More precisely, we determine the set of $\b$ for which $\MMw\ne\emptyset$ and show that $\MMw$ is irreducible for these $\b$. Our approach is to apply Oprea's stacky version \cite{Oprea} of a theorem of Bia\l{ynicki}-Birula \cite{BB} to the $\ct$-action on $\MMw$ defined in Section \ref{3c} and study the geometry of a particular fixed-point component. It turns out that this component is determined by $\MM_{0,2}(G/P,\b')$ for some other $\b'$ and $\fm{\cm}{G/P}$ defined in Section \ref{3d}.

\begin{definition} \label{3efix} Let $\CC^{\times}$ act on a smooth Deligne-Mumford stack $\XXX$. A component $\FF$ of $\XXX^{\CC^{\times}}$ is said to be \textit{attractive} if for a (and hence any) geometric point $x\in \FF$, the weights (more precisely, the orbi-weights) of the tangent space $T_x\XXX$ with respect to the $\CC^{\times}$-action are all non-negative.
\end{definition}

\begin{theorem}\label{3ebb} \cite{BB} Let $X$ be a smooth quasi-projective variety with a $\CC^{\times}$-action and $F$ an attractive component of $X^{\CC^{\times}}$. There exists a unique $\CC^{\times}$-invariant open subscheme $U$ of $X$ containing $F$ which is isomorphic to a $\CC^{\times}$-equivariant affine fibration over $F$.
\end{theorem}

For other components of $X^{\CC^{\times}}$, there are similar affine fibrations which are in general locally closed subschemes. Bia\l{ynicki}-Birula also showed that if $X$ is proper, these subschemes form a decomposition of $X$. His result has been generalized by Oprea \cite{Oprea} to Deligne-Mumford stacks. For our purpose, we only need the following application.

\begin{theorem} \label{3ebbstack} Let $\XXX$ be a non-empty proper smooth Deligne-Mumford stack with a $\CC^{\times}$-action. Suppose $\XXX$ admits a $\CC^{\times}$-equivariant \'etale atlas. Then $\XXX^{\CC^{\times}}$ has an attractive component. It is unique if and only if $\XXX$ is irreducible.
\end{theorem}


Denote by $\fm{\ct}{\Gw}$ the unique attractive component of $\Gw^{\ct}$.
\begin{lemma}\label{3elemma} $f_{\wl}(\fm{\ct}{\Gw})\subseteq G\cdot x_{\wll}$.
\end{lemma}
\begin{proof}
Let $\g\in \Gw$ be the unique point such that $f_{\wl}(\g)=x_{\wll}$. The result follows from the observations that $\BB^{0,-}\cdot \g$ is open and the weights of $\lie(B^-)$ (resp. $\lie(\BB^{0,-})/ \lie(B^-)$) are all zero (resp. positive).
\end{proof}

By Lemma \ref{3dtrans}, $\Gt{P_{\wll}}\fm{\cw}{G/P}$ is well-defined. Consider the diagram

\begin{equation}\label{3edia}
\begin{tikzpicture}
\tikzmath{\x1 = 5; \x2 = 2.5; \x3=1.8;}
\node (A) at (-\x1,0) {$\fm{\ct}{\Gw}$} ;
\node (B) at (-\x2,-\x3) {$G\cdot x_{\wll}\simeq G/P_{\wll}$} ;
\node (C) at (0,0) {$\Gt{P_{\wll}}\fm{\cw}{G/P}$} ;
\node (D) at (\x2,-\x3) {$G/P$};

\node at (-1.4*\x2, -0.4*\x3) {{\scriptsize $f$}};
\node at (-0.58*\x2, -0.4*\x3) {{\scriptsize $\pi$}};
\node at (0.58*\x2, -0.4*\x3) {{\scriptsize $j$}};
\path[->, font=\scriptsize]
(A) edge node[anchor = south west]{} (B);
\path[->, font=\scriptsize]
(C) edge node[anchor = south east]{} (B);
\path[->, font=\scriptsize]
(C) edge node[anchor = south west]{} (D);

\end{tikzpicture}
\end{equation}

where
\begin{itemize}
\item $\fm{\cw}{G/P}$ and $P_{\wll}$ are defined in Section \ref{3d};
\item $f$ is the restriction of $f_{\wl}$ to $\fm{\ct}{\Gw}$ (it does land in $G\cdot x_{\wll}$ by Lemma \ref{3elemma});
\item $\pi$ is the canonical projection; and
\item $j$ is the unique $G$-equivariant map extending the inclusion $\fm{\cw}{G/P}\hookrightarrow G/P$.
\end{itemize}

\begin{definition} \label{3edef} Define a smooth variety 
\[F_{\wl} := \fm{\ct}{\Gw}\times_{(f,\pi)}(\Gt{P_{\wll}}\fm{\cw}{G/P})\]
and $h_{\wl}:F_{\wl}\ra G/P$ to be the composite
\[ F_{\wl}\xrightarrow{f'}\Gt{P_{\wll}}\fm{\cw}{G/P}\xrightarrow{j} G/P \]
where $f'$ is induced by $f$ in the fiber product.
\end{definition}

The role of $F_{\wl}$ is to parametrize a class of holomorphic sections of $\fib{\wl}|_{\PP^1\times \fm{\ct}{\Gw}}$. Notice $\fm{\cw}{G/P}$ parametrizes a component of the space of constant sections of $\fibc{{\wll}}$. Indeed, $\fm{\cw}{G/P}$ is attractive so we have
\begin{equation}\label{3eimply}
\ds \lim_{z\to 0}x_{\mu}(z^{-1})\cdot y \in \fm{\cw}{G/P} ~\Longrightarrow ~y\in \fm{\cw}{G/P}.
\end{equation}
\noindent Since $\fib{}|_{\PP^1\times (G\cdot x_{\wll})}\simeq \Gt{P_{\wll}}\fibc{{\wll}}$, we see that $F_{\wl}$ parametrizes the pullbacks of the $G$-translates of these constant sections.

\begin{definition} \label{3eembdef}
Let $\b\in\pi_2(G/P)$. Define a morphism
\begin{equation}\label{3eemb}
F_{\wl}\times_{(h_{\wl},\ev_1)}\MM_{0,2}(G/P,\b-\b_{\wll})\ra \MMw
\end{equation}
as follows. 
\begin{itemize}
\item Every point of the domain of \eqref{3eemb} is of the form $(\g,[g:y],u)$ where 
\begin{enumerate}[(a)]
\item $\g\in \fm{\ct}{\Gw}$;

\item $[g:y]\in \Gt{P_{\wll}}\fm{\cw}{G/P}$; and

\item $u\in \MM_{0,2}(G/P,\b-\b_{\wll})$
\end{enumerate}

\noindent such that $f_{\wl}(\g)=g\cdot x_{\wll}$ and $u(z_1)=g\cdot y$, where $z_1$ is the first marked point on the domain of $u$. 

\item We send this point to $u_1\#u_2$ where
\begin{enumerate}[(i)]
\item $u_1$ is the section of $\fib{\wl}|_{\PP^1\times\{\g\}}\simeq \fib{}|_{\PP^1\times\{g\cdot x_{\wll}\}}$ which is the $g$-translate of the constant section of $\fibc{{\wll}}$ corresponding to $y$; and

\item $u_2$ is just $u$ but regarded as a stable map to the fiber of $\fib{\wl}$ over $(\infty,\g)$.
\end{enumerate}

\item If $\b=\b_{\wll}$, the domain of \eqref{3eemb} is understood to be $F_{\wl}$. In this case, the morphism is defined in a similar way.
\end{itemize}
\end{definition}

It is easy to see that morphism \eqref{3eemb} is injective. One can show that it is even a closed immersion. But we will not use this fact.

\begin{proposition}\label{3emain} Suppose $\MMw\ne\emptyset$. Then $\MMw^{\ct}$ has a unique attractive component. Set-theoretically, it is equal to the image of morphism \eqref{3eemb}. 
\end{proposition}

Recall $\ec\subset\pi_2(G/P)$ is the semigroup of effective curve classes in $G/P$ and $\b_{\wll}\in \pi_2(G/P)$ is defined in Definition \ref{3dbwl}.
\begin{corollary}\label{3ecor1} The set $\{\b\in\pi_2(G/P)|~\MMw\ne\emptyset\}$ is equal to $\b_{\wll}+\ec$.
\end{corollary}
\begin{proof}
If $\b=\b_{\wll}$, then the domain of \eqref{3eemb} is $F_{\wl}$ which is clearly non-empty. If $\b\ne\b_{\wll}$, then the stack $\MM_{0,2}(G/P,\b-\b_{\wll})$ is non-empty if and only if $\b\in\b_{\wll}+(\ec\setminus\{0\})$.
\end{proof}

\begin{corollary}\label{3ecor2} For any $\b\in \b_{\wll}+\ec$, $\MMw$ is irreducible. 
\end{corollary}
\begin{proof}
This follows from Proposition \ref{3emain} and Theorem \ref{3ebbstack}. Notice that $\MMw$ does admit a $\ct$-equivariant \'etale atlas, provided we reparametrize the torus $\ct$. (So we actually apply Theorem \ref{3ebbstack} to this reparametrized action, but this will not affect our arguments.) See Remark \ref{4brmk}.
\end{proof}

\begin{myproof}{Proposition}{\ref{3emain}}
First observe that the domain of \eqref{3eemb} is irreducible, by a result of Kim-Pandharipande \cite{KP}, and that this morphism sends every point into $\MMw^{\ct}$. It follows that the image of \eqref{3eemb} is contained in a unique component of $\MMw^{\ct}$. Let $\FF$ be an attractive component of $\MMw^{\ct}$ which exists by Theorem \ref{3ebbstack}. Let $u\in \FF$. Observe that $u$ lies over a point $p$ in a component $F$ of $\Gw^{\ct}$.

We claim $F=\fm{\ct}{\Gw}$. Suppose the contrary. Since the $T$-action commutes with the $\ct$-action (see Section \ref{3c}(2)(i)), $T$ preserves $\FF$, and hence we can replace $u$ with another $\ol{u}\in\FF$ which is also a $T$-fixed point of $\MMw$. Then $\ol{u}$ lies over a $T$-fixed point $\ol{p}$ of $F$. Since $F\ne \fm{\ct}{\Gw}$, the tangent space $T_{\ol{p}}\Gw$ contains a weight vector $v$ of negative weight. Consider $\BB^{0,-}\cdot\ol{p}$, the $\BB^{0,-}$-orbit passing through $\ol{p}$. Then $v\not\in T_{\ol{p}}(\BB^{0,-}\cdot\ol{p})$. By \cite[Proposition 4.5]{me}, $\ol{u}$ is still unobstructed when it is regarded as a stable map to $\fib{\wl}|_{\PP^1\times(\BB^{0,-}\cdot \ol{p})}$ (a smooth $\BB^{0,-}$-equivariant compactification of $\BB^{0,-}\cdot\ol{p}$ is not required since $\ol{u}$ is $T$-invariant). Therefore, $v$ lifts to a weight vector in the tangent space $T_{\ol{u}}\MMw$ which has the same weight as $v$. By assumption, the weight is negative, a contradiction.

By Lemma \ref{3elemma}, $f_{\wl}(\fm{\ct}{\Gw})\subseteq G\cdot x_{\wll}$ and 
\[\fib{\wl}|_{\PP^1\times\{p\}}\simeq \fib{}|_{\PP^1\times{\{f_{\wl}(p)}\}}\simeq \fibc{{\wll}}\]
as $\ct$-varieties. Since the rest of the proof relies only on constructing some deformation vector fields of $u$ in $\fib{\wl}|_{\PP^1\times\{p\}}$, we may assume $f_{\wl}(p)=x_{\wll}$ so that $u$ is a stable map to $\fibc{{\wll}}$ which represents a section class and is a $\ct$-fixed point in the moduli. Write $u=u_0\#u_s\#u_{\infty}$ where $u_s$ is a section and $u_0$ (resp. $u_{\infty}$) is a stable map to the fiber of $\fibc{{\wll}}$ over $0$ (resp. $\infty$). By Lemma \ref{3dconst}, $u_s$ is the constant section corresponding to a point $y\in G/P$.

We first reduce the situation to the case $y\in (G/P)^{\cw}$. More precisely, we show that there exists another stable map $u'\in\FF$ such that if we write  $u'=u'_0\#u'_s\#u'_{\infty}$ as before, then $u'_s$ is the constant section corresponding to a point in $(G/P)^{\cw}$. For any $\eta\in\CC^{\times}$, define $y_{\eta}:= x_{\mu}(\eta^{-1})\cdot y$. Let $u_{y_{\eta}}$ be the constant section of $\fib{\wll}$ corresponding to $y_{\eta}$. Define a morphism 
\begin{equation}\label{3emor} 
\begin{array}{ccc}
\CC^{\times}& \ra & \MMw\\ [.5em] 
\eta &\mapsto & u_0\# u_{y_{\eta}}\# (x_{\mu}(\eta^{-1})\cdot u_{\infty})
\end{array}
\end{equation}
By \cite[Proposition 6]{FP}, after a base change $\CC^{\times}\ra \CC^{\times}$, the above morphism extends to a morphism defined on $\CC$. The stable map at $\eta =0$ will be our $u'$.

From now on, we assume $y\in (G/P)^{\cw}$. We show 
\begin{enumerate}
\item $u_0$ does not exist; and
\item $y\in\fm{\cw}{G/P}$.
\end{enumerate}
For (1), recall $\ct$ acts on $\PP^1$, the base of $\fibc{{\wll}}$, in the following way:
\[ t\cdot z = \left\{ 
\begin{array}{cl}
t^{-1}z & 0\text{-chart}\\
tz & \infty\text{-chart}
\end{array}
\right. .\]
Take a weight vector $\zeta\in H^0(\PP^1;T\PP^1)$ such that $\zeta(0)\ne 0$. Then it has weight $-1$. It is easy to show that there exists a weight vector $\zeta'$ in $H^0(u^*T\fibc{{\wll}})$ which is non-tangential to $u$ such that $\zeta'|_{u_s}$ projects to $\zeta$. It follows that $\zeta'$ defines a weight vector in $T_u\MMw$ of weight $-1$, a contradiction. 

For (2), notice that $T_y(G/P)$ is isomorphic, as $\cw$-modules, to a direct sum of weight spaces of the form $\CC_{\a(\wll)}$ where $\a\in R$. If $y\not\in\fm{\cw}{G/P}$, then $T_y(G/P)$ contains a weight vector $v\in\CC_{\a(\wll)}$ for some $\a$ with $\a(\wll)<0$. Define a vector field $\xi\in H^0(\PP^1; u_s^*T\fibc{{\wll}})$ by  
\[ \xi(z) = \left\{ 
\begin{array}{cl}
v & 0\text{-chart}\\
z^{-\a(\wll)} v& \infty\text{-chart}
\end{array}
\right. .\]
Since $\xi(\infty)=0$, we can extend $\xi$ trivially to a deformation vector field $\xi'$ of $u$. (Recall we have proved that $u_0$ does not exist.) It is clear that $\xi'$ is non-tangential to $u$, and hence it defines a weight vector in $T_u\MMw$ of weight $\a(\wll)<0$, a contradiction.

Thus, every $u\in\FF$ is contained in the set-theoretic image of morphism \eqref{3eemb} after passing to the limit of morphism \eqref{3emor}. By \eqref{3eimply}, $u$ actually lies in the image set before passing to the limit. The proof of Proposition \ref{3emain} is complete. 
\end{myproof}

\section{Proof of the main theorem}\label{4}
\subsection{Construction of the homomorphism} \label{4a} 
\begin{definition}\label{4adef} 
Define an $\rep$-linear map
\[\Phi: \bigoplus_{\mu\in\Q} \fof(\rep)\langle\localb_{\mu}\rangle\ra QK_T(G/P)[\ec^{-1}]\otimes\fof(\rep)\]
by
\[\Phi(\localb_{\mu}):= \sum_{i,j\in I} \sum_{\b\in\pi_2(G/P)} q^{\b} g^{ij}\chi_{\MMm}(\OO_{\MMm}^{vir}\otimes \ev^*e_i)(\id+\agp)^{-1}(e_j),\]
where $\OO_{\MMm}^{vir}\in K^T(\MMm)$ is the virtual structure sheaf constructed in \cite{Lee}. See Section \ref{2b} for the definition of $\{e_i\}_{i\in I}$, $\{g^{ij}\}_{i,j\in I}$ and $\agp$, and Section \ref{3b} for the definition of $\MMm$. 
\end{definition}
In order for $\Phi$ to be well-defined, we must verify
\begin{lemma} $\Phi(\localb_{\mu})$ lands in $QK_T(G/P)[\ec^{-1}]$.
\end{lemma}
\begin{proof}
This follows from Corollary \ref{3ecor1} since $\MMm\subseteq\MMw$ for some $\wl\in W_{af}^-$.
\end{proof}

\begin{proposition}\label{4aprop} $\Phi$ is an $\rep$-algebra homomorphism.
\end{proposition}
\begin{proof}
The proof relies heavily on Appendix \ref{app1}. Let $\mu_1,\mu_2\in\Q$. We have 
\[\Phi(\localb_{\mu_1}\bulletsmall \localb_{\mu_2})= \Phi(\localb_{\mu_1+\mu_2}) = \sum_{i,j\in I} g^{ij}\pgw_{\mu_1+\mu_2}(e_i) (\id+\agp)^{-1}(e_j)\]
where $\pgw$ is defined in Definition \ref{appdef2}. By Proposition \ref{appmain}, the last expression is equal to 
\begin{align*}
& \sum_{i,j\in I}\sum_{i',j'\in I}g^{ij} g^{i'j'} \pgw_{\mu_1}\left(e_i,(\id+\agp)^{-1}(e_{i'})\right)\pgw_{\mu_2}(e_{j'}) (\id+\agp)^{-1}(e_j)\\
=~& \sum_{i,j\in I} g^{ij}\pgw_{\mu_1}\left(e_i,\Phi(\localb_{\mu_2})\right)  (\id+\agp)^{-1}(e_j).
\end{align*}
Applying Proposition \ref{appmain} again, to the splitting $\mu_1=\mu_1+0$, the last expression is equal to
\begin{align*}
& \sum_{i,j\in I}\sum_{i',j'\in I}g^{ij} g^{i'j'} \pgw_{\mu=0}\left(e_i,\Phi(\localb_{\mu_2}),(\id+\agp)^{-1}(e_{i'})\right)\pgw_{\mu_1}(e_{j'}) (\id+\agp)^{-1}(e_j)\\
=~& \sum_{i,j\in I}g^{ij} \pgw_{\mu=0}\left(e_i,\Phi(\localb_{\mu_2}),\Phi(\localb_{\mu_1}) \right) (\id+\agp)^{-1}(e_j).
\end{align*}
But we have (cf. \cite[Lemma 3.7]{me})
\[ \pgw_{\mu=0}\left(e_i,\Phi(\localb_{\mu_2}),\Phi(\localb_{\mu_1}) \right) = \sum_{\b\in\ec}q^{\b} \gw^{\b}(\Phi(\localb_{\mu_1}),\Phi(\localb_{\mu_2}),e_i). \]
Therefore, 
\[ \Phi(\localb_{\mu_1}\bulletsmall \localb_{\mu_2})=\Phi(\localb_{\mu_1})\star\Phi(\localb_{\mu_2})\]
as desired.
\end{proof}
\subsection{Step 1 of the computation: reducing to the initial term} \label{4b} 
The main result of this subsection is Proposition \ref{4bprop}. Let $\wl\in W_{af}^-$. Recall the map $f_{\wl}:\Gw\ra \GG/\GG^0$ defined in Section \ref{3b}. By abuse of notation, we also denote by $f_{\wl}$ the composite
\[ \Gw\xrightarrow{f_{\wl}} \GG/\GG^0 \ra L_{sm}G/L_{sm}^0G\xrightarrow{\sim} \O_{sm} K\]
where the second arrow is the obvious map and the third is the diffeomorphism in Theorem \ref{2cdiffeo}.
\begin{lemma} $f_{\wl}$ is a $\BB^{0,-}$-equivariant resolution of $\ol{\BB^{0,-}_{sm}\cdot x_{\wll}}$. In particular, the image of $f_{\wl}$ lies in $Gr_G=\O_{pol}K$.
\end{lemma}
\begin{proof}  
First notice that for any $n\in\NN$ and $i=0,\ldots, r$, we have 
\[ \PPc_{\widetilde{\a}_i}\cdot Gr^{(n)}(H)\subseteq Gr^{(n+1)}(H).\]
It follows that $f_{\wl}$ lands in $Gr^{(N)}(H)$ for sufficiently large $N$. In particular, $f_{\wl}$ is algebraic. Since $\wl\in W_{af}^-$ and the word defining $\Gw$ is reduced, there exists a unique point $\g\in\Gw$ such that $f_{\wl}(\g)=x_{\wll}$. Moreover, the orbit $\BB^{0,-}\cdot\g$ is open and $f_{\wl}|_{\BB^{0,-}\cdot\g}$ is bijective onto $\BB^{0,-}_{sm}\cdot x_{\wll}$. The rest of the proof is clear. 
\end{proof}

\begin{lemma} \label{4bratsing} $(f_{\wl})_*[\OO_{\Gw}] = [\OO_{\ol{\BB^{0,-}_{sm}\cdot x_{\wll}}}]=\ascl_{\wl}\in K^T(Gr_G)$.
\end{lemma}
\begin{proof}
This follows from the fact that affine Schubert varieties have rational singularities. A proof can be found in \cite[Theorem 8.2.2]{Kumarbook}. Although our definition of these varieties is a priori different from the one in \textit{loc. cit.}, the arguments there apply well to our case.
\end{proof}

Recall $\ascl_{\wl}$ is regarded as an element of the domain of $\Phi$ via \eqref{2cmono}.
\begin{lemma} \label{4bchangebasis} We have 
\[ \Phi(\ascl_{\wl})=\sum_{i,j\in I} \sum_{\b\in\b_{\wll}+\ec} q^{\b} g^{ij}\chi_{\MMw}(\ev^*e_i)(\id+\agp)^{-1}(e_j)\]
where $\b_{\wll}$ is defined in Definition \ref{3dbwl}.
\end{lemma}
\begin{proof}
By the classical localization formula,
\[ [\OO_{\Gw}] = \sum_{\g\in \Gw^T}\frac{1}{\Lambda_{-1}((T_{\g}\Gw)^{\vee})} [\OO_{\g}]\in K^T(\Gw)\]
where $\Lambda_{-1}(V):=\sum_{i\geqslant 0} (-1)^i[\Lambda^i V]\in\rep$ for any $T$-module $V$. Applying $(f_{\wl})_*$ to both sides of the last equation and using Lemma \ref{4bratsing}, we get
\[  \ascl_{\wl} = \sum_{\mu\in\Q} \left(\sum_{\g\in\Gw^T\cap f_{\wl}^{-1}(\mu)} \frac{1}{\Lambda_{-1}((T_{\g}\Gw)^{\vee})}\right)\localb_{\mu} .\] 
The rest follows from a parallel argument used in the proof of \cite[Lemma 3.9]{me} which deals with the case of quantum cohomology. In our case, we need the $K$-theoretic version of virtual localization formula in \cite{GP}. See \cite{Kvirtual} for the explicit formula and its proof. 
\end{proof}

\begin{remark}\label{4brmk}
Some care needs to be taken when we apply the virtual localization formula: The proof of this formula given in \cite{Kvirtual} assumes an extra condition which, by \cite[Proposition 5.13]{KS}, is satisfied if our stack admits a $T$-equivariant \'etale atlas of finite type. According to the remark following that proposition, which cites \cite[Theorem 4.3]{AHR}, every separated Deligne-Mumford stack of finite type with a $T$-action admits such an atlas, after possibly reparametrizing $T$. Notice that such reparametrization will not affect the argument in the proof of Lemma \ref{4bchangebasis}. Alternatively, the existence of the required atlas for our particular stack follows from \cite[Corollary 4]{Oprea}.
\end{remark}

Define 
\[\vp:= \sum_{i,j\in I}\sum_{\b\in\b_{\wll}+\ec} q^{\b} g^{ij} \chi_{\MMw}(\ev^*e_i)e_j.\]
Notice the absence of $(\id+\agp)^{-1}$. Write $\vp=\vp_0+\vp_+$ where $\vp_0$ (resp. $\vp_+$) is the expression contributed by $\b=\b_{\wll}$ (resp. $\b\ne\b_{\wll}$). 
 
\begin{proposition}\label{4bprop} $\Phi(\ascl_{\wl})=\vp_0$.
\end{proposition}
\begin{proof}
By Lemma \ref{4blem1} below, we have $\vp_+=\agp(\vp_0)$, and hence
\[\Phi(\ascl_{\wl}) = (\id +\agp)^{-1}(\vp_0+\vp_+)= (\id +\agp)^{-1}\circ(\id +\agp)(\vp_0) = \vp_0.\]
\end{proof}

Before proving Lemma \ref{4blem1} which is used in the proof of Proposition \ref{4bprop}, we first prove another lemma. Recall morphism \eqref{3eemb}. Its domain is $F_{\wl}\times_{(h_{\wl},\ev_1)}\MM_{0,2}(G/P,\b-\b_{\wll})$ if $\b\ne\b_{\wll}$ and $F_{\wl}$ if $\b=\b_{\wll}$.
\begin{lemma} \label{4blem2}
We have
\[\ev_*[\OO_{\MMw}]=\left\{
\begin{array}{cc}
(\ev'_2)_* [\OO_{F_{\wl}\times_{(h_{\wl},\ev_1)}\MM_{0,2}(G/P,\b-\b_{\wll})}] & \b\ne\b_{\wll}\\ [1em]
(h_{\wl})_* [\OO_{F_{\wl}}] & \b=\b_{\wll}
\end{array}\right.\]
where $\ev'_2$ is induced by the evaluation map on $\MM_{0,2}(G/P,\b-\b_{\wll})$ at the second marked point.
\end{lemma} 
\begin{proof}
By Remark \ref{2abrmk}, it suffices to verify the corresponding equality for their coarse moduli. Let $M$ and $F$ be the coarse moduli of $\MMw$ and the domain of \eqref{3eemb}. They are projective varieties with only finite quotient singularities. Since $\MMw$ is irreducible by Corollary \ref{3ecor2} and the canonical morphism $\MMw\ra M$ is surjective, $M$ is also irreducible. Similarly, $F$ is irreducible. By the definition of coarse moduli, morphism \eqref{3eemb} induces a unique morphism $\iota:F\ra M$ which is equal to \eqref{3eemb} set-theoretically. The uniqueness implies that $\iota$ is $T$-equivariant.

Denote by $M^{sm}$ the smooth locus of $M$. It is easy to see that $\iota^{-1}(M^{sm})\ne\emptyset$ (look at chains of embedded spheres) and $\iota$ maps $\iota^{-1}(M^{sm})$ bijectively onto an attractive component $F'$ of $(M^{sm})^{\ct}$. Since $\iota^{-1}(M^{sm})$ is reduced, $\iota|_{\iota^{-1}(M^{sm})}$ factors through the inclusion $F'\hookrightarrow M^{sm}$. There exists a non-empty open subscheme $V\subseteq \iota^{-1}(M^{sm})$ such that $\iota|_V$ is smooth over $F'$. Since $\iota$ is injective, $\iota|_V$ is \'etale over $F'$ and hence an isomorphism onto its image. 

By Theorem \ref{3ebb}, there exists a $\ct$-invariant open subscheme $U\subseteq M^{sm}$ containing $F'$ and an affine fibration $U\ra F'$. The latter morphism induces, via the morphism $\iota|_V:V\xrightarrow{\sim}\iota(V)\subseteq F'$, a rational map $\phi:M\dashrightarrow F$. Since $\iota$ is $T$-equivariant, $U$ is unique and the $T$-action commutes with the $\ct$-action (see Section \ref{3c}(2)(i)), it follows that everything is $T$-equivariant. 

By resolving the indeterminacy locus of $\phi$, we obtain a smooth irreducible projective variety $Z$ and morphisms $\nu_M:Z\ra M$ and $\nu_F:Z\ra F$ such that $\nu_M$ is birational and $\phi\circ \nu_M=\nu_F$. Since equivariant resolutions of singularities exist (see e.g. \cite{equivresolution}), we may assume $Z$ has a $T$-action and $\nu_M$, $\nu_F$ are $T$-equivariant. Define $\ev_F:=\ev_2'$ if $\b\ne\b_{\wll}$ and $\ev_F:=h_{\wl}$ otherwise. By the fact that $\ev$ is $\ct$-invariant, we have $\ev=\ev_F\circ\phi$, and hence $\ev\circ\nu_M=\ev_F\circ \nu_F$, giving
\[ \ev_*(\nu_M)_*[\OO_Z] = (\ev_F)_*(\nu_F)_*[\OO_Z]\in K_T(G/P).\]

Since $Z$ has only finite quotient singularities and singularities of this kind are rational, by \cite{fqimpliesr}, we have $(\nu_M)_*[\OO_Z]=[\OO_M]$. To conclude the proof, it suffices to show $(\nu_F)_*[\OO_Z]=[\OO_F]$. This follows from \cite[Theorem 3.1]{BM}, given the following conditions:
\begin{enumerate}
\item $\nu_F$ is surjective and $T$-equivariant;
\item $Z$ and $F$ are projective with rational singularities; and
\item the general fiber of $\nu_F$ is rational. 
\end{enumerate}
Condition (1) is obvious. To verify (2), we use the above cited result \cite{fqimpliesr}. For (3), take a non-empty open subscheme $W\subseteq V$ such that $\nu_F|_{\nu_F^{-1}(W)}$ is smooth. Since $Z$ contains an open dense subscheme $U'$ such that $\nu_F|_{U'}$ is an affine fibration over $W$, it follows that every geometric fiber of $\nu_F|_{\nu_F^{-1}(W)}$ is connected and contains the affine space as an open subscheme, i.e. it is rational.
\end{proof}

\begin{lemma} \label{4blem1}
$\vp_+=\agp(\vp_0)$.
\end{lemma}
\begin{proof} By the projection formula,
\[ \vp_+= \sum_{\b\in\b_{\wll}+(\ec\setminus\{0\})}q^{\b}\ev_*[\OO_{\MMw}]\]
which is equal to $\sum_{\b\in \b_{\wll}+(\ec\setminus\{0\})} q^{\b} (\ev'_2)_*[\OO_{F_{\wl}\times_{(h_{\wl},\ev_1)}\MM_{0,2}(G/P,\b-\b_{\wll})}]$ by Lemma \ref{4blem2}. 

Denote by $\ev_1$ and $\ev_2$ the evaluation maps on $\MM_{0,2}(G/P,\b-\b_{\wll})$. Since $\ev_1$ is flat, we have
\begin{equation}\label{4beq} 
(\ev'_2)_*[\OO_{F_{\wl}\times_{(h_{\wl},\ev_1)}\MM_{0,2}(G/P,\b-\b_{\wll})}] = (\ev_2)_*(\ev_1)^*(h_{\wl})_*[\OO_{F_{\wl}}],
\end{equation}
by the base change formula. Summing up \eqref{4beq} over all $\b\in\b_{\wll}+(\ec\setminus\{0\})$, weighted by $q^{\b}$, we get $\vp_+=\agp(q^{\b_{\wll}}(h_{\wl})_*[\OO_{F_{\wl}}])$. By Lemma \ref{4blem2} applied to $\b=\b_{\wll}$, we get $\vp_0=q^{\b_{\wll}}(h_{\wl})_*[\OO_{F_{\wl}}]$. The result follows. 
\end{proof}
\subsection{Step 2 of the computation: determining the initial term} \label{4c} 
By Proposition \ref{4bprop}, it suffices to determine $\vp_0$. By Lemma \ref{4blem2} applied to $\b=\b_{\wll}$, we have 
\begin{equation}\label{4ceq1} 
\vp_0 =q^{\b_{\wll}}(h_{\wl})_*[\OO_{F_{\wl}}] 
\end{equation}
where $F_{\wl}$ and $h_{\wl}$ are defined in Definition \ref{3edef}. Since $\fm{\ct}{\Gw}$ is $B^-$-invariant and $f=f_{\wl}|_{\fm{\ct}{\Gw}}$ is $B^-$-equivariant, $f$ is birational onto the Schubert variety $\ol{B^-\cdot x_{\wll}}\subseteq G\cdot x_{\wll}\simeq G/P_{\wll}$. Since Schubert varieties have rational singularities (see Lemma \ref{4bratsing}), we have, by the base change formula, 
\begin{equation}\label{4ceq2}
(h_{\wl})_*[\OO_{F_{\wl}}]=j_*[\OO_{\pi^{-1}(\ol{B^-\cdot x_{\wll}})}].
\end{equation}

\noindent See \eqref{3edia} for the definition of $j$ and $\pi$.

Let us deal with the case $P=B^+$ first. Define $y_w':=wB^+\in G/B^+$ and $B_w:= wBw^{-1}$. 
\begin{lemma}\label{4clem2} $~$
\begin{enumerate}
\item We have $y_w'\in\fm{\cw}{G/B^+}$ and $B_w\subseteq P_{\wll}$.

\item There exists a $G$-equivariant isomorphism 
\[ \Gt{P_{\wll}}\fm{\cw}{G/B^+}\simeq G/B_w\]
under which $\pi$ and $j$ (see diagram \eqref{3edia}) are identified with the projection $G/B_w\ra G/P_{\wll}$ and the isomorphism $G/B_w\xrightarrow{\sim} G/B^+: gB_w\mapsto gwB^+$ respectively. 

\item $j$ maps $\pi^{-1}(\ol{B^-\cdot x_{\wll}})$ isomorphically onto $\ol{B^-\cdot y_w'}$. 
\end{enumerate}
\end{lemma}
\begin{proof}
(1) is proved by looking at the weight spaces:
\[ T_{y_w'}(G/B^+)\simeq \bigoplus_{\a\in -wR^+}\gg_{\a}\quad\text{ and }\quad \lie(B_w)=\hh\oplus\bigoplus_{\a\in wR^+}\gg_{\a}.\]
Since $\wl\in W_{af}^-$, $\llll$ is anti-dominant, and hence $\pm\a(\wll)\geqslant 0$ for any $\a\in \mp wR^+$. 

To prove (2), recall (Lemma \ref{3dtrans}) $P_{\wll}$ acts transitively on $\fm{\cw}{G/B^+}$. Hence, by (1), we have $\fm{\cw}{G/B^+}\simeq P_{\wll}/B_w$ so that 
\[ \Gt{P_{\wll}}\fm{\cw}{G/B^+}\simeq \Gt{P_{\wll}}(P_{\wll}/B_w)\simeq G/B_w.\]
The rest of the proof is clear.

For (3), we use the identifications in (2). Denote by $C$ the dominant Weyl chamber. It suffices to show that $wC$ has the smallest length (with respect to $C$) among other chambers which contain $-\wll$. It amounts to showing 
\begin{equation}\label{4bsmallest}
\a(\wll)=0\text{ and }\a\in R^+\Longrightarrow w^{-1}\a\in R^+.
\end{equation}
This requires the assumption $\wl\in W_{af}^-$. Denote by $\Delta$ the dominant alcove. By definition, the alcove $\wl(\Delta)=\wll+w(\Delta)$ has the smallest length (with respect to $\Delta$) among other alcoves which contain $\wll$. If $\a(\wll)=0$, then $\wll+w(\Delta)$ and $\Delta$ lie in the same side with respect to the wall $\{\a=0\}$. This proves \eqref{4bsmallest}.
\end{proof}

For general $P$, we have the commutative diagram 
\begin{center}

\vspace{.2cm}
\begin{tikzpicture}
\tikzmath{\x1 = 3; \x2 = 1.5; \x3=3.2;}
\node (A) at (0,0) {$G\cdot x_{\wll}$} ;
\node (B) at (-\x1,-\x2) {$\Gt{P_{\wll}}\fm{\cw}{G/B^+}$} ;
\node (C) at (\x1,-\x2) {$\Gt{P_{\wll}}\fm{\cw}{G/P}$} ;
\node (D) at (-\x1,-\x3) {$G/B^+$};
\node (E) at (\x1,-\x3) {$G/P$};
\node (F) at (-0.88*\x1,-0.5*\x3-0.55*\x2) {$\simeq$};

\node at (-0.6*\x1, -0.3*\x2) {{\scriptsize $\pi_{G/B^+}$}};
\node at (0.58*\x1, -0.3*\x2) {{\scriptsize $\pi_{G/P}$}};
\path[->, font=\scriptsize]
(B) edge node[anchor = south]{} (C);
\path[->, font=\scriptsize]
(D) edge node[anchor = south]{$p$} (E);
\path[->, font=\scriptsize]
(B) edge node[anchor = south east]{} (A);
\path[->, font=\scriptsize]
(C) edge node[anchor = south west]{} (A);
\path[->, font=\scriptsize]
(B) edge node[left]{$j_{G/B^+}$} (D);
\path[->, font=\scriptsize]
(C) edge node[right]{$j_{G/P}$} (E);
\end{tikzpicture}
\end{center}
where the horizontal arrows are some canonical projections. By Lemma \ref{4clem2} and the fact that the upper horizontal arrow has rational fibers, we have
\begin{equation}\label{4ceq3}
(j_{G/P})_*[\OO_{(\pi_{G/P})^{-1}(\ol{B^-\cdot x_{\wll}})}] = p_*[\OO_{\ol{B^-\cdot y_w'}}].
\end{equation}

Denote by $\widetilde{w}\in W/W_P$ the minimal length coset representative of $wW_P$. Notice $p(\ol{B^-\cdot y_w'})=\ol{B^-\cdot y_{\widetilde{w}}}$ but the dimension of some fibers of $p|_{\ol{B^-\cdot y_w'}}$ may be positive. Choose Bott-Samelson resolutions $\G'\ra \ol{B^-\cdot y_w'}$ and $\G\ra \ol{B^-\cdot y_{\widetilde{w}}}$ such that there is a map $p':\G'\ra \G$ defined by forgetting last few factors of $\G'$ which fits into the commutative diagram 
\begin{equation}\nonumber
\begin{tikzcd} 
 \G'  \arrow{d}[left]{p'} \arrow{r} &[1.0em] \ol{B^-\cdot y_w'} \arrow{d}{p}   \\
 \G \arrow{r}  &[1.0em]  \ol{B^-\cdot y_{\widetilde{w}}}
\end{tikzcd}.
\end{equation}

\noindent By the facts that $p'$ has rational fibers and Schubert varieties have rational singularities, we have 
\begin{equation}\label{4ceq4} 
p_*[\OO_{\ol{B^-\cdot y_w'}}]=[\OO_{\ol{B^-\cdot y_{\widetilde{w}}}}].
\end{equation}

\noindent Alternatively, \eqref{4ceq4} follows from \cite[Theorem 3.1]{BM}.
 
\subsection{Proof of Theorem \ref{main}} \label{4d} 
We start with summarizing what we have done in the previous subsections. We defined an $\rep$-linear map $\Phi$ in Definition \ref{4adef} and proved in Proposition \ref{4aprop} that it is an $\rep$-algebra homomorphism. A priori, $\Phi$ was defined in terms of the localization basis $\{\localb_{\mu}\}_{\mu\in\Q}$. But by Lemma \ref{4bchangebasis}, it can also be defined in the same style in terms of the affine Schubert basis $\{\ascl_{\wl}\}_{\wl\in W_{af}^-}$. By Proposition \ref{4bprop}, \eqref{4ceq1}, \eqref{4ceq2}, \eqref{4ceq3} and \eqref{4ceq4}, we have
\[ \Phi(\ascl_{\wl})= q^{\b_{\wll}} [\OO_{\ol{B^-\cdot y_{\widetilde{w}}}}].\]

It remains to show that $\b_{\wll}$ corresponds to $\llll+\Q_P$ via the dual of isomorphism \eqref{2bisom}. By making use of the canonical projection $G/B^+\ra G/P$, we may assume $P=B^+$. Let $\rho\in (\Q)^*$. We have to show $\deg(u^*\LL_{\rho})=\rho(\llll)$ for the constant section $u$ of $\fibb{\wll}$ corresponding to a point of $\fm{\cw}{G/B^+}$. By Lemma \ref{4clem2}, we can take that point to be $y'_w:=wB^+$. Recall $L_{\rho}=\Gt{B^+}\CC_{-\rho}$ so that $(L_{\rho})_{y'_w}\simeq \CC_{-w\rho}$ as $T$-modules, and hence $(L_{\rho})_{y'_w}\simeq \CC_{-\rho(\llll)}$ as $\cw$-modules. Therefore, $u^*\LL_{\rho}\simeq\OO_{\PP^1}(\rho(\llll))$ as desired. The proof of Theorem \ref{main} is complete. 


\appendix  
\section{A K-theoretic degeneration formula} \label{app1}
Let $\mu\in\Q$. Recall the $G/P$-fibration $\fib{\mu}$ over $\PP^1$ and the projection $\pi_{\mu}:\fib{\mu}\ra \PP^1$ defined in Section \ref{3b}. Let $k\in\NN$. Fix some points $z_1,\ldots,z_k\in\PP^1$. For each $i=1,\ldots,k$, define $D_{\mu,i}:=\pi_{\mu}^{-1}(z_i)$ and $\iota_{\mu,i}:D_{\mu,i}\hookrightarrow \fib{\mu}$ to be the inclusion.
\begin{definition}\label{appdef1} Let $\b\in\pi_2(G/P)$. 
\begin{enumerate}
\item Define 
\[ \MM_k(\mu,\b):= \bigcup_{\widetilde{\b}}\MM_{0,k}(\fibc{{\mu}},\widetilde{\b})\times_{(\vec{\ev},\iota_{\mu,1}\times\cdots\times\iota_{\mu,k})}\left(D_{\mu,1}\times\cdots\times D_{\mu,k}\right)\]
where $\widetilde{\b}$ runs over all classes in $\pi_2(\fibc{{\mu}})$ satisfying
\begin{enumerate}[(i)]
\item $(\pi_{\mu})_*\widetilde{\b}=[\PP^1]\in\pi_2(\PP^1)$; and 
\item $\langle\widetilde{\b},c_1(\LL_{\rho})\rangle = \langle\b,c_1(L_{\rho})\rangle$ for any $\rho\in(\Q/\Q_P)^*$.
\end{enumerate}
(The line bundles $\LL_{\rho}$ and $L_{\rho}$ are defined in Section \ref{3a} and Section \ref{2b} respectively.)

\vspace{.2cm}
\item By abuse of notation, define 
\[\ev_i:\MM_k(\mu,\b)\ra D_{\mu,i}\simeq G/P\]
to be the morphism induced by the evaluation map $\ev_i$ on $\MM_{0,k}(\fibc{{\mu}},\widetilde{\b})$.
\end{enumerate}
\end{definition}

\begin{definition}\label{appdef2}$~$
\begin{enumerate}
\item Let $\g_1,\ldots,\g_k\in K_T(G/P)$ and $\b\in\pi_2(G/P)$. Define
\[\pgw^{\b}_{\mu}(\g_1,\ldots,\g_k) := \chi_{\MM_k(\mu,\b)}\left(\OO^{vir}_{\MM_k(\mu,\b)}\otimes \ev_1^*\g_1\otimes\cdots \otimes \ev_k^*\g_k \right)\in R(T)\]
where $\OO^{vir}_{\MM_k(\mu,\b)}$ is the virtual structure sheaf constructed in \cite{Lee}.

\item Let $\g_1,\ldots,\g_k\in K_T(G/P)$. Define
\[ \pgw_{\mu}(\g_1,\ldots,\g_k) := \sum_{\b\in\pi_2(G/P)} q^{\b}\pgw^{\b}_{\mu}(\g_1,\ldots,\g_k) \in R(T)\otimes\ZZ[[\ec,\ec^{-1}].  \]
(By Corollary \ref{3ecor1}, the ring $\ZZ[[\ec,\ec^{-1}]$ is large enough in order for $\pgw_{\mu}$ to be well-defined.)
\end{enumerate}
\end{definition}

Recall the notations $\{e_i\}_{i\in I}$, $\{g^{ij}\}_{i,j\in I}$ and $\agp$ defined in Section \ref{2c}. The following proposition is a $K$-theoretic degeneration formula which we need in order to prove Proposition \ref{4aprop}.
\begin{proposition}\label{appmain}
For any $\mu_1,\mu_2\in\Q$ and tuples $\vec{\g}^{(1)}=(\g^{(1)}_1,\ldots, \g^{(1)}_{k_1})$, $\vec{\g}^{(2)}=(\g^{(2)}_1,\ldots, \g^{(2)}_{k_2})$ of elements of $K_T(G/P)$,
\[ \pgw_{\mu_1+\mu_2}(\vec{\g}^{(1)},\vec{\g}^{(2)}) = \sum_{i,j\in I}g^{ij}\pgw_{\mu_1}\left(\vec{\g}^{(1)},(\id+\agp)^{-1}(e_i)\right) \pgw_{\mu_2}(\vec{\g}^{(2)},e_j).\]
\end{proposition}

The rest of the appendix is devoted to the proof of Proposition \ref{appmain}. While it should not be difficult to prove the result using the machineries developed by Li \cite{Li1,Li2}, we take a more direct approach, namely we work with the moduli of stable maps to the degeneration family instead of working with the stack of expanded degenerations and its associated moduli of relative stable maps.

Define 
\[C:=\{ (t,[x:y:z])\in\AA\times\PP^2|~xy=tz^2\}.\]
Then the projection onto $\AA$ defines a flat family $C\ra\AA$. Denote by $C_t$ the fiber over $t\in\AA$. If $t\ne 0 $, $C_t$ is isomorphic to $\PP^1$, and if $t=0$, it is isomorphic to the nodal rational curve $C_1\cup C_2$ where $C_1, C_2\simeq \PP^1$. In \cite{me}, we constructed a locally trivial $G/P$-fibration 
\[p:\dfib\ra C\]
such that 
\[\dfib|_{C_{t\ne 0}}\simeq \fib{\mu_1+\mu_2}~\text{ and }~\dfib|_{C_i}\simeq \fib{\mu_i}.\]
By the construction, $\dfib$ admits a $T$-action such that $p$ is $T$-equivariant where $T$ acts $C$ trivially. Moreover, the above isomorphisms are $T$-equivariant. Let $\rho\in (\Q/\Q_P)^*$. Like $\fib{\mu}$, there is a line bundle on $\dfib$ which restricts to $L_{\rho} =G\times_P \CC_{-\rho} $ on each fiber of $p$. By abuse of notation, we denote this line bundle by $\LL_{\rho}$.

Fix some sections $s_1^{(1)},\ldots,s_{k_1}^{(1)}, s_1^{(2)},\ldots,s_{k_2}^{(2)}$ of the family $C\ra \AA$ such that $s_1^{(i)}(0),\ldots,s_{k_i}^{(i)}(0)$ lie in $C_i\subset C_{t=0}$ away from the intersection $C_1\cap C_2$. For each $s=1,\ldots, k_1$, the divisor $D_{s}^{(1)}:=\dfib\times_{(p,s_s^{(1)})}\AA$ of $\dfib$ is canonically isomorphic to $\AA\times G/P$. Denote by $\iota_s^{(1)}:D_s^{(1)}\hookrightarrow \dfib$ the inclusion and $j_s^{(1)}:D_s^{(1)}\ra G/P$ the projection. For each $t=1,\ldots, k_2$, we define $D_t^{(2)}$, $\iota_t^{(2)}$ and $j_t^{(2)}$ in a similar way. 

From now on, fix $\b\in\pi_2(G/P)$. We have to show 
\begin{equation}\label{appeq1}
\pgw^{\b}_{\mu_1+\mu_2}(\vec{\g}^{(1)},\vec{\g}^{(2)})=\sum_{r=0}^{\infty}\sum_{i,j\in I}\sum_{\vec{\abc}}(-1)^r g^{ij}\pgw^{\abc_1}_{\mu_1}(\vec{\g}^{(1)}, I^{\abc_{11},\ldots, \abc_{1r}}(e_i)) \pgw^{\abc_2}_{\mu_2}(\vec{\g}^{(2)},e_j).
\end{equation}
Here,
\begin{itemize}
\item the third sum runs over all tuples $\vec{\abc}=(\abc_1;\abc_{11},\ldots,\abc_{1r};\abc_2)$ of elements of $\pi_2(G/P)$ such that $\abc_{1i}\ne 0$ for all $i=1,\ldots, r$ and $\abc_1+\sum_{i=1}^r\abc_{1i}+\abc_2=\b$; and

\item $I^{\abc_{11},\ldots, \abc_{1r}}$ is a linear operator on $K_T(G/P)$ defined by
\[ I^{\abc_{11},\ldots, \abc_{1r}} := (\ev_2^{\abc_{11}})_*\circ (\ev_1^{\abc_{11}})^* \circ \cdots\circ  (\ev_2^{\abc_{1r}})_*\circ (\ev_1^{\abc_{1r}})^* \] 
where $\ev_1^{\b'}, \ev_2^{\b'}: \MM_{0,2}(G/P,\b')\ra G/P$ are the evaluation maps. 
\end{itemize}

Denote by $A$ the semigroup of effective curve classes in $\dfib$. Let $\ac\in A$ and $k\in\NN$. Consider the Artin stack $\AMM_{0,k,A,\ac}$ defined in \cite{Costello}. Roughly speaking, it parametrizes genus zero $k$-pointed prestable curves each of whose irreducible components is assigned an element of $A$ such that the sum of these elements is equal to $\ac$ and this assignment satisfies a stability condition. There is a natural morphism $\MM_{0,k}(\dfib,\ac)\ra \AMM_{0,k,A,\ac}$ defined by forgetting the target but remembering the degrees of the restrictions of any stable maps to the irreducible components of the domain curves. It is shown in \textit{loc. cit.} that $\AMM_{0,k,A,\ac}$ is \'etale over the usual Artin stack $\AMM_{0,k}$ of prestable curves. It follows that $\AMM_{0,k,A,\ac}$ is smooth and the standard virtual structure sheaf $\OO^{vir}_{\MM_{0,k}(\dfib,\ac)}$ is equal to the one constructed using the relative perfect obstruction theory associated to the above morphism.

\begin{definition}\label{appdef3}$~$
\begin{enumerate}
\item Define a Deligne-Mumford stack
\[\MM:=\bigcup_{\widetilde{\b}}\MM_{0,k_1+k_2}(\dfib,\widetilde{\b}) \]
where $\widetilde{\b}$ runs over all classes in $\pi_2(\dfib)$ satisfying
\begin{enumerate}[(i)]
\item $p_*\widetilde{\b}=[C_1]+[C_2]$, the fiber class of the flat family $C\ra \AA$; and

\item $\langle\widetilde{\b},c_1(\LL_{\rho})\rangle = \langle\b,c_1(L_{\rho})\rangle$ for any $\rho\in(\Q/\Q_P)^*$. 
\end{enumerate}

\item Define an Artin stack
\[\AMM_A:=\bigcup_{\widetilde{\b}}\AMM_{0,k_1+k_2,A,\widetilde{\b}}\]
where $\widetilde{\b}$ runs over the same set in (1) above.
\end{enumerate}
\end{definition}

There is a proper morphism $\pi:\MM\ra\AA$ sending each stable map $u$ to the unique $t\in\AA$ such that $p\circ u$ lands in $C_t$. There is also a natural morphism $\nu:\MM\ra\AMM_A$ which is the union of the morphisms $\MM_{0,k_1+k_2}(\dfib,\widetilde{\b})\ra\AMM_{0,k_1+k_2,A,\widetilde{\b}}$ mentioned above. 

\begin{definition}\label{appdef4} Let $r$ be a non-negative integer.
\begin{enumerate}
\item Define $\O^r$ to be the set of modular graphs with $A$-structure \cite{BManin} satisfying
\begin{enumerate}[(a)]
\item the number of vertices is $r+2$;

\item the genus associated to every vertex is zero;

\item the underlying graph without tails is a linear graph;

\item the degrees associated to the two end vertices of the graph in (c) are classes which are projected via $p$ to $[C_1]$ and $[C_2]$ respectively;

\item the sum of degrees satisfies conditions (i) and (ii) in the definition of $\MM$ above; and

\item the number of tails is $k_1+k_2$. 
\end{enumerate}

\item Define $\widetilde{\O}^r$ to be the set of modular graphs with $A$-structure satisfying (a) to (e) above and that the number of tails is zero.
\end{enumerate}
\end{definition}

It follows from (d) and (e) that for any $\s\in \O^r$ or $\widetilde{\O}^r$, the degrees associated to the intermediate vertices are fiber classes with respect to the fibration $p:\dfib\ra C$. 

Define an injective map $\widetilde{\O}^r\hookrightarrow \O^r: \s\mapsto \ol{\s}$ as follows. For each $\s\in\widetilde{\O}^r$, define $\ol{\s}$ to be the modular graph obtained from $\s$ by attaching the first $k_1$ tails to the end vertex corresponding to $C_1$ and the rest to the other end vertex. See (d) in the definition of $\O^r$. We will identify $\widetilde{\O}^r$ with the image of this injective map. 

For any $\s\in\O^r$, we have a similarly-defined Deligne-Mumford stack $\MM(\s)$ (resp. Artin stack $\AMM_A(\s)$) parametrizing stable maps (resp. prestable curves) which are at least singular as described by the graph $\s$. If $\s\in \O^0$, then $\AMM_A(\s)$ is a Weil divisor of $\AMM_A$. Since $\AMM_A$ is smooth, we have the associated Cartier divisor $\OO_{\AMM_A}(\AMM_A(\s))$ (for simplicity, the section is hidden from the notation). For each $\ell\in\NN$, define $\AMM^{\leqslant\ell}_{0,0}$ to be the moduli stack of genus zero prestable curves with at most $\ell$ nodes. We have a smooth divisor $\DDD$ of $\AMMo$ which is the complement of the open substack $\AMM^{\leqslant 0}_{0,0}$. 

\begin{lemma} \label{applem} 
There exists a commutative diagram of stacks
\begin{equation}\label{appdiag1}
\begin{tikzcd} 
\MM  \arrow{d}[left]{\pi} \arrow{r}{\nu} &[1.0em] \AMM_A \arrow{d}[right]{g}   \\
\AA \arrow{r}{f}  &[1.0em]  \AMMo
\end{tikzcd}
\end{equation}
such that $f^*\OO_{\AMMo}(\DDD)=\OO_{\AA}(\mathbf{0})$ and $g^*\OO_{\AMMo}(\DDD)=\bigotimes_{\s\in\O^0}\OO_{\AMM_A}(\AMM_A(\s))$ as Cartier divisors.
\end{lemma}
\begin{proof}
We postpone the proof until the end.
\end{proof}

Observe that $\bigcup_{\s\in\O^0}\AMM_A(\s)\subset\AMM_A$ is a normal crossing divisor and the intersection of these divisors $\AMM_A(\s)$ over any finite subset of $\O^0$ is of the form $\AMM_A(\s')$ for some modular graph $\s'$ with $A$-structure. It is easy to see that in our case these modular graphs are precisely the elements of $\O^r$, $r\in\NN$. Therefore, by the inclusion-exclusion principle \cite{Lee}, 
\begin{equation}\label{appeq3}
\left[\OO_{\bigcup_{\s\in\O^0}\AMM_A(\s)}\right] =\sum_{r=0}^{\infty}(-1)^r\sum_{\s\in\O^r}\left[\OO_{\AMM_A(\s)} \right].
\end{equation}
By Lemma \ref{applem}, \eqref{appeq3} and some standard arguments in virtual pullbacks \cite{Kvirtual}, we have
\begin{equation}\label{appeq4}
\text{LHS of }\eqref{appeq1} = \sum_{r=0}^{\infty}(-1)^r\sum_{\s\in\O^r}\chi_{\MM(\s)}\left(\OO^{vir}_{\MM(\s)}\otimes e_1\otimes e_2 \right)
\end{equation}
where $e_1:= \bigotimes_{s=1}^{k_1}\ev_s^*(\iota_s^{(1)})_*(j_s^{(1)})^*\g_s^{(1)}  $ and $e_2:= \bigotimes_{t=1}^{k_2}\ev_{k_1+t}^*(\iota_t^{(2)})_*(j_t^{(2)})^*\g_t^{(2)} $.

Let $\s\in\O^r$. Recall the subset $\widetilde{\O}^r\subseteq\O^r$ defined above. Suppose $\s\not\in\widetilde{\O}^r$. It is easy to see that one of the evaluation maps on $\MM(\s)$ factors through the open immersion $\dfib\setminus D_s^{(1)}\hookrightarrow \dfib$ or $\dfib\setminus D_t^{(2)}\hookrightarrow \dfib$ for some $s=1,\ldots,k_1$ or $t=1,\ldots,k_2$. It follows that the summand in \eqref{appeq4} is zero for $\s$. Therefore, we have
\begin{equation}\label{appeq5}
\text{RHS of }\eqref{appeq4} = \sum_{r=0}^{\infty}(-1)^r\sum_{\s\in\widetilde{\O}^r}\chi_{\MM(\ol{\s})}\left(\OO^{vir}_{\MM(\ol{\s})}\otimes e_1\otimes e_2 \right).
\end{equation}

Now let $\s\in\widetilde{\O}^r$. Denote the degrees of $\s$ by $\ac_1^{\s}, \ac_{11}^{\s},\ldots, \ac_{1r}^{\s}, \ac_2^{\s}$, in the order compatible with the linear graph, starting with the end vertex corresponding to $C_1$. Define $\ifib:=\dfib|_{C_i}$, $F:=\dfib|_{C_1\cap C_2}$ and $\iota_F^{(i)}:F\hookrightarrow \ifib$ to be the inclusion. For each $i=1,2$, define 
\[ \MM_i:= \MM_{0,k_i+1}(\ifib,\ac_i^{\s}) \times_{(\ev_{k_i+1},\iota_F^{(i)})} F\]
and $\ev_{(i)}:\MM_i\ra F$ to be the morphism induced by $\ev_{k_i+1}$. Then Proposition \ref{appmain} follows if we can show
\begin{align}
& \chi_{\MM(\ol{\s})}\left( \OO^{vir}_{\MM(\ol{\s})}\otimes e_1\otimes e_2\right) \nonumber \\
=~& \sum_{i,j\in I}g^{ij}\chi_{\MM_1}\left( \OO^{vir}_{\MM_1}\otimes e_1\otimes \ev_{(1)}^*\left(I^{\ac_{11}^{\s},\ldots,\ac_{1r}^{\s}}(e_i)\right)\right)\chi_{\MM_2}\left(\OO^{vir}_{\MM_2}\otimes e_2\otimes \ev_{(2)}^*e_j\right). \label{appeq6}
\end{align}

\noindent The last equation follows from a similar argument which is used to prove the ``cutting edges'' axiom in \cite{B}. More precisely, we prove
\begin{lemma}
There exists a Cartesian diagram
\begin{equation}\label{appdiag2}
\begin{tikzcd} 
\MM(\ol{\s}) \arrow{d} \arrow{r} &[1.0em] \MM_1\times\displaystyle\left(\prod_{i=1}^r\MM_{0,2}(F,\ac^{\s}_{1i})\right)\times\MM_2 \arrow{d}  \\
 F^{r+1} \arrow{r}{\Delta_F^{r+1}}  &[1.0em]  F^{2r+2}
\end{tikzcd}
\end{equation}
such that the obstruction theories on the two stacks in the upper row are compatible over $\Delta_F^{r+1}$. 
\end{lemma}
\begin{proof}
Denote by $\CCC_1, \CCC_{11},\ldots,\CCC_{1r},\CCC_2$ the universal curves of $\MM(\ol{\s})$ corresponding to different vertices of $\ol{\s}$, and by $u_1,u_{11},\ldots, u_{1r},u_2$ the universal stable maps. Observe that over any $\CC$-point of $\MM(\ol{\s})$, $u_1$, $u_{1i}$ and $u_2$ factor through scheme-theoretically the inclusions of $\afib$, $F$ and $\bfib$ into $\dfib$ respectively. Using the fact that $C_1$ and $C_2$ are $(-1)$-curves in $C$ and the theorem of cohomology and base change, these stable maps (over $\MM(\ol{\s})$) in fact land in these subschemes. This defines the horizontal arrow in \eqref{appdiag2}. The vertical arrows are defined to be the evaluation maps at those marked points which do not correspond to any of the insertions $\vec{\g}^{(1)}$ and $\vec{\g}^{(2)}$. It is straightforward to verify that \eqref{appdiag2} is Cartesian.

It remains to verify the compatibility of obstruction theories. Let $\CCC$ denote the universal curve of $\MM(\ol{\s})$ obtained by gluing $\CCC_1,\CCC_{11},\ldots,\CCC_{1r},\CCC_2$ according to the combinatorics of $\ol{\s}$. Let $u:\CCC\ra \dfib$ be the universal stable map, $\eta:\CCC_1\amalg\cdots\amalg\CCC_2\ra \CCC$ the gluing map, and $z_0,\ldots,z_r:\MM(\ol{\s})\ra \CCC$ the sections ``lying between'' $\CCC_1$ and $\CCC_{11}$, $\CCC_{11}$ and $\CCC_{12}$, etc. For each $i=1,2$, we remove the component $\CCC_{3-i}$ from $\CCC$, that is, we glue all $\CCC_1,\CCC_{11},\ldots,\CCC_{1r},\CCC_2$ except $\CCC_{3-i}$. Denote by $\widetilde{\CCC}_i$ the resulting curve and by $\eta_i:\widetilde{\CCC}_i\hookrightarrow\CCC$ the inclusion. 

Define $\FF_i$ to be the kernel of the natural epimorphism 
\[ u_i^*T_{\ifib}\ra (x_i)_*(u_i\circ x_i)^*N_{F/P_{\mu_i}}\ra 0\]
where $x_1:\MM(\ol{\s})\ra \CCC_1$ (resp. $x_2:\MM(\ol{\s})\ra \CCC_2$) is the section which is used to glue $\CCC_1$ and $\CCC_{11}$ (resp. $\CCC_{1r}$ and $\CCC_2$). (Recall $u_i$ lands in $\ifib$.) Notice that $\FF_i$ is the sheaf giving rise to the deformation and obstruction spaces for $\MM_i$ relative to $\AMM_A(\ol{\s})$. It is straightforward to verify that there exists a coherent sheaf $\mathcal{E}$ on $\CCC$ which fits into the following two short exact sequences simultaneously:
\begin{equation}\label{appseq1}
0\ra \mathcal{E}\ra \eta_*\left(\FF_1\oplus \bigoplus_{i=1}^r u_{1i}^*T_F \oplus\FF_2\right) \ra \bigoplus_{j=0}^r(z_j)_*(u\circ z_j)^*T_F\ra 0 
\end{equation}
\begin{equation}\label{appseq2}
0\ra \mathcal{E}\ra u^*T_{\dfib} \ra \bigoplus_{i=1}^2 (\eta_i)_*(u\circ \eta_i)^* N_{\ifib/\dfib }\ra 0.
\end{equation}

Since $(u\circ \eta_i)^* N_{\ifib/\dfib }\simeq (p\circ u\circ\eta_i)^*\OO_{C_i}(-1)$, the derived pushforward functor sends the cokernel in \eqref{appseq2} to the zero object in the derived category $D(\MM(\ol{\s}))$ and so defines an isomorphism between the objects corresponding to the other two coherent sheaves in the same short exact sequence. Using this isomorphism and applying the argument in \cite{B} to \eqref{appseq1}, we obtain a homomorphism of distinguished triangles
\begin{equation}\nonumber
\begin{tikzcd} 
  h^*E_{\MM_{right}}\arrow{d} \arrow{r} &[.5em] E_{\MM_{left}} \arrow{d} \arrow{r} &[.5em] v^*L_{\Delta^{r+1}_F}\arrow{d} \arrow{r} &[.5em]  h^*E_{\MM_{right}}[1] \arrow{d} \\
 h^*L_{\MM_{right}/\AMM_A(\ol{\s})}\arrow{r} &[.5em] L_{\MM_{left}/\AMM_A(\ol{\s})} \arrow{r} &[.5em] L_{\MM_{left}/\MM_{right}}\arrow{r} &[.5em]  h^*L_{\MM_{right}/\AMM_A(\ol{\s})}[1] 
\end{tikzcd}
\end{equation}
where $\MM_{left}$ (resp. $\MM_{right}$) is the stack at the top left (resp. right) corner in diagram \eqref{appdiag2}, and $h:\MM_{left}\ra\MM_{right}$ and $v:\MM_{left}\ra F^{r+1}$ are the morphisms in the same diagram. This is the compatibility we need to prove.
\end{proof}

Equation \eqref{appeq6} now follows from the functoriality of virtual structure sheaves \cite{Lee}. The proof is Proposition \ref{appmain} is complete.

\vspace{.2cm}
\begin{myproof}{Lemma}{\ref{applem}} By making use of the ``forgetting tails'' morphisms for $\MM$ and $\AMM_A$, we may assume $k_1=k_2=0$. Notice that the latter morphism exists by \cite[Proposition 2.1.1]{Costello}. Define $f$ to be the morphism representing the flat family $C\ra \AA$ defined at the beginning.

We define $g$ as follows. Denote by $B$ the semigroup of effective curve classes in $C$. The projection $p:\dfib\ra C$ induces a homomorphism $\phi:A\ra B$ of semigroups. There is a morphism $\AMM_A\ra \AMM_{0,0,B,[C_1]+[C_2]}$ sending each curve $S$ over $\spec\CC$ to the curve obtained by replacing the degree $a\in A$ assigned to each irreducible component with $\phi(a)\in B$ and then stabilizing the resulting curve in the sense of \cite{Costello}, i.e. contracting those irreducible components labelled by $0\in B$ and having less than three special points. Although the construction of this morphism is standard, we provide the details in the next paragraph.

Let $T$ be a scheme and $S\ra T$ a family of prestable curves representing a $T$-point of $\AMM_{0,0,A,a}$. By definition, we have a constructible function $\deg:S\ra A$ which records the degree assigned to each irreducible component of the fiber over any $\CC$-point of $T$. We have to stabilize the family $S\ra T$ with respect to the degree function $\phi\circ\deg :S\ra B$. As in the usual case, this type of stabilization is universal so that we can construct it \'etale locally. \'Etale locally, we can make $S$ stable in the usual sense by inserting sufficiently many sections. Then we remove these sections one by one. Each time we remove a section, we stabilize the family by contracting those components labelled by $0\in B$ and having two special points. Such a stabilization does exist, by \cite[Proposition 2.1.1]{Costello} which says that $\CCC_{0,N,B,b}\simeq \AMM_{0,N+1,B,b}$ where $\CCC_{0,N,B,b}$ is the universal curve of $\AMM_{0,N,B,b}$. The family obtained after removing all these sections will be the desired stabilization. 

Our $g$ is then defined to be the composite
\[ \AMM_A \ra \AMM_{0,0,B,[C_1]+[C_2]}\ra \AMMo\]
where the first morphism is the above morphism and the second is defined by forgetting all the degrees without stabilizing. It is easy to see that $g$ does land in $\AMMo$.

\vspace{.2cm}
It remains to verify the stated properties. 

\vspace{.2cm}
\noindent \textit{The commutativity of diagram \eqref{appdiag1}.}

Let $T$ be a scheme and $u:S\ra \dfib$ a stable map over $T$ representing a $T$-point of $\MM$. By \cite[Lemma 2.2]{BManin}, $u$ lies over a unique morphism $v:T\ra \AA$. By definition, $f\circ\pi$ sends this $T$-point to the $T$-point of $\AMMo$ represented by the flat family $C\times_{\AA} T$. Consider the morphism $w:S\ra C\times_{\AA} T$ induced by $p\circ u$. Observe that, over every $\CC$-point of $T$, $w$ contracts precisely those irreducible components which are mapped into the fibers of $p$. It follows that $C\times_{\AA} T$ is the stabilization of $p\circ u$ which is equal to the $T$-point with respect to the composite $g\circ\nu$. 

\vspace{.4cm}
\noindent \textit{The equality $f^*\OO_{\AMMo}(\DDD)=\OO_{\AA}(\mathbf{0})$.}

This is straightforward. For example, one can make use of the quotient stack presentation of $\AMMo$ given in \cite{Artinonenode}. More precisely, $\AMMo$ is isomorphic to $[V/GL_3]$ where $V\subset \sym^2\CC^3$ is the space of symmetric 2-tensors on $\CC^3$ with rank at least two, and the divisor $\DDD$ corresponds to the space of those tensors whose rank is precisely two.

\vspace{.4cm}
\noindent \textit{The equality $g^*\OO_{\AMMo}(\DDD)=\bigotimes_{\s\in\O^0}\OO_{\AMM_A}(\AMM_A(\s))$.}

Observe that, in a smooth atlas, the equality $g^{-1}(\DDD)=\bigcup_{\s\in\O^0}\AMM_A(\s)$ holds set-theoretically. It follows that $g^*\OO_{\AMMo}(\DDD)=\bigotimes_{\s\in\O^0}\OO_{\AMM_A}(\AMM_A(\s))^{\otimes m_{\s}}$ as Cartier divisors for some $m_{\s}\in\ZZ_{\geqslant 1}$. To determine $m_{\s}$, consider the morphism $h_{\s}:\AA\ra \AMM_A$ which represents the flat family $C\ra\AA$ where $C_1$ and $C_2$ are labelled by the degrees given by $\s$. By the previous paragraph, we have $h_{\s}^*g^*\OO_{\AMMo}(\DDD)=\OO_{\AA}(\mathbf{0})$. On the other hand, it is easy to see that for any $\s'\in\O^0$ not equal to $\s$, $h_{\s}^*\OO_{\AMM_A}(\AMM_A(\s'))$ is the trivial Cartier divisor. This gives $m_{\s}=1$. \end{myproof}

\end{document}


